\documentclass[11pt]{article}
\usepackage{amssymb}
\usepackage{amsmath}
\usepackage{mathrsfs}
\usepackage{enumerate}
\usepackage{graphics}
\usepackage{graphicx}
\usepackage{dsfont}
\usepackage{color}
\usepackage{subfigure}
\usepackage[T1]{fontenc}
\usepackage{latexsym,amssymb,amsmath,amsfonts,amsthm}
\usepackage{txfonts}
\newcommand{\R}{{\mathbb R}}
\newcommand{\Z}{{\mathbb Z}}

\newcommand{\Sp}{{\mathbb S}}

\newcommand{\no}{\nonumber}
\newcommand{\be}{\begin{eqnarray}}
\newcommand{\ben}{\begin{eqnarray*}}
\newcommand{\en}{\end{eqnarray}}
\newcommand{\enn}{\end{eqnarray*}}
\newcommand{\ba}{\backslash}
\newcommand{\pa}{\partial}

\newcommand{\ov}{\overline}

\newcommand{\I}{{\rm Im}}
\newcommand{\Rt}{{\rm Re}}
\newcommand{\g}{\gamma}
\newcommand{\G}{\Gamma}

\newcommand{\Om}{\Omega}

\newcommand{\la}{\lambda}

\newcommand{\ol}{\overline}

\newcommand{\half}{\frac{1}{2}}

\newtheorem{theorem}{Theorem}[section]

\newtheorem{corollary}[theorem]{Corollary}

\newtheorem{remark}[theorem]{Remark}

\newtheorem{algorithm}{Algorithm}[section]

\begin{document}
\renewcommand{\theequation}{\arabic{section}.\arabic{equation}}

\begin{titlepage}
\title{\bf Recovering scattering obstacles by multi-frequency phaseless far-field data}
\author{Bo Zhang\thanks{LSEC and Institute of Applied Mathematics, AMSS, Chinese Academy of Sciences,
Beijing, 100190, China and School of Mathematical Sciences, University of Chinese Academy of Sciences,
Beijing 100049, China ({\tt b.zhang@amt.ac.cn})}
\and Haiwen Zhang\thanks{Institute of Applied Mathematics, AMSS, Chinese Academy of Sciences,
Beijing 100190, China ({\tt zhanghaiwen@amss.ac.cn})}
}
\date{}
\end{titlepage}
\maketitle

\begin{abstract}
It is well known that the modulus of the far-field pattern (or phaseless far-field pattern) is invariant under translations
of the scattering obstacle if only one plane wave is used as the incident field, so the shape but not the location of the
obstacle can be recovered from the phaseless far-field data. This paper aims to devise an approach to break
the translation invariance property of the phaseless far-field pattern. To this end, we make use of the superposition of
two plane waves rather than one plane wave as the incident field. In this paper, it is mathematically proved that
the translation invariance property of the phaseless far-field pattern can indeed be broken if superpositions of two plane
waves are used as the incident fields for all wave numbers in a finite interval.
Furthermore, a recursive Newton-type iteration algorithm in frequencies is also developed to numerically recover both
the location and the shape of the obstacle simultaneously from multi-frequency phaseless far-field data.
Numerical examples are also carried out to illustrate the validity of the approach and the effectiveness of
the inversion algorithm.
\end{abstract}

\section{Introduction}\label{se1}

Inverse acoustic and electromagnetic scattering by bounded obstacles has a wide application in many areas
such as radar, sonar, remote sensing, geophysics, medical imaging and nondestructive testing, and therefore
has been extensively studied mathematically and numerically (see, e.g. \cite{CK13}).
This paper is devoted to the recovery of scattering obstacles from phaseless far-field data.
Our discussion is restricted to the two-dimensional case by assuming that the obstacle is an infinite cylinder
which is invariant in the $x_3$ direction.

To give a precise description of the problem, assume that the obstacle $\Om\subset\R^2$
is a bounded domain with $C^2-$smooth connected boundary $\G:=\pa\Om$.
Denote by $\nu$ the unit outward normal on $\G$ and by $\Sp^1$ the unit circle.
Consider the time-harmonic ($e^{-i\omega t}$ time dependence) plane wave
\ben
u^i=u^i(x;d,k):=\exp(ikd\cdot x)
\enn
which is incident on the obstacle $\Om$ from the unbounded part $\R^2\ba\ov{\Om}$,
where $d=(\cos\theta,\sin\theta)^T\in\Sp^1$ is the incident direction and $\theta\in[0,2\pi]$ is the incidence angle.
Here, $k=\omega/c>0$ is the wave number, $\omega$ and $c$ are the wave frequency and speed in $\R^2\ba\ov{\Om}$,
respectively. Then the total field $u=u^i+u^s$ is given as the sum of the incident field $u^i$
and the unknown scattered field $u^s$. For impenetrable obstacles, the scattered field $u^s$ in $\R^2\ba\ov{\Om}$
satisfies the problem
\begin{align}
\label{eq1}\Delta u^s+k^2 u^s=0&\quad \textrm{in}\;\;\R^2\ba\ov{\Om}\\
\label{eq2}\mathscr{B}(u^s)=f&\quad \textrm{on}\;\;\G\\
\label{eq3}\lim_{r\to\infty}r^\half\left(\frac{\pa u^s}{\pa r}-ik u^s\right)=0&\quad r=|x|.
\end{align}
Here, (\ref{eq2}) represents the Dirichlet condition with $\mathscr{B}(u^s):=u^s$ (that is, $\Om$ is a sound-soft obstacle),
the Neumann condition with $\mathscr{B}(u^s):={\pa u^s}/{\pa\nu}$ (corresponding to a sound-hard obstacle $\Om$) or
an impedance boundary condition with $\mathscr{B}(u^s):={\pa u^s}/{\pa\nu}+\mu u^s$, where $\mu$ is the surface impedance
coefficient depending on the physical property of the obstacle $\Om$ (in this paper, we assume that $\mu$ is a constant),
so the corresponding boundary data $f=-u^i|_\G$, $f=-({\pa u^i}/{\pa\nu})|_\G$ or $f=-({\pa u^i}/{\pa\nu}+\mu u^i)|_\G$,
respectively. Further, for penetrable obstacles, the scattered field $u^s$ in $\R^2\ba\ol{\Om}$ and the transmitted field $u$
in $\Om$ satisfy the problem
\begin{align}
\label{eq4}\Delta u^s+k^2 u^s=0&\quad \textrm{in}\;\;\R^2\ba\ol{\Om}\\
\label{eq5}\Delta u+k^2 nu=0&\quad \textrm{in}\;\;\Om\\
\label{eq6}u^s_+ -u_-=f_1,\quad
\frac{\pa u^s_+}{\pa\nu}-\lambda\frac{\pa u_-}{\pa\nu}=f_2 &\quad \textrm{on}\;\;\G\\
\label{eq7}\lim_{r\to\infty}r^\half\left(\frac{\pa u^s}{\pa r}-ik u^s\right)=0&\quad r=|x|
\end{align}
where (\ref{eq6}) is called the transmission boundary condition with the boundary
data $(f_1,f_2)=(-u^i|_\G,-({\pa u^i}/{\pa\nu})|_\G)$.
Here, $n$ is the index of refraction of the medium in $\Om$ which is assumed a constant with $\Rt(n)>0$ and $\I(n)\ge0$,
$\la$ is the transmission parameter which is a positive constant and depends on the property of the medium in $\R^2\ba\ol{\Om}$
and $\Om$, and the subscripts $+$ and $-$ denote the limits from the exterior and interior of the boundary, respectively.
Note that (\ref{eq3}) and (\ref{eq7}) are the well-known Sommerfeld radiation condition.

The well-posedness of the problems (\ref{eq1})-(\ref{eq3}) and (\ref{eq4})-(\ref{eq7}) have been
established by the variational method \cite{CC14} or the integral equation method \cite{CK83,CK13}.
In addition, it is known that the scattered field $u^s$ has the following asymptotic behavior
\be\label{asymp-1}
u^s(x;d,k)=\frac{e^{ik|x|}}{\sqrt{|x|}}\left(u^\infty(\hat{x};d,k)+O\Big(\frac{1}{|x|}\Big)\right),
\qquad |x|\to\infty,
\en
uniformly for all observation directions $\hat{x}=x/|x|$ on $\Sp^1$, where $u^\infty$ is
called the far-field pattern of the scattered field $u^s$.

Many inversion algorithms have been developed for the numerical reconstruction of the obstacle $\Om$
or its boundary $\G$ from the far-field pattern (see, e.g. the iteration methods \cite{K93} and the singular sources
method \cite{P01,PS05}, the linear sampling method \cite{CC14} and the factorization method \cite{KG08}).
An excellent and comprehensive survey on these results can be found in the monograph \cite{CK13}.
{
In addition, some recent works on inverse scattering can be found in \cite{AKKL12,AGKLS12,AGJK12},
where a time-reversal, a linearization approach (with respect to the size of the obstacle/operating wavelength),
a topological derivative based imaging functionals are introduced and their resolution and
stability with respect to measurement and medium noises are quantified.
}

In many practical applications, the phase of the far-field pattern can not be measured accurately
compared with its modulus, and therefore it is often desirable to reconstruct the scattering obstacle
from the modulus of the far-field pattern or the phaseless far-field data.
However, not many results are available for such problems both mathematically and numerically.
Kress and Rundell first studied such inverse problems in \cite{KR97} and proved that
for the case when $\Om$ is a sound-soft obstacle the modulus $|u^\infty(\hat{x};d,k)|$ of the far-field pattern
is invariant under translations of the obstacle $\Om$, that is, for the shifted domain
$\Om_\ell:=\{x+\ell\,:\,x\in\Om\}$ with $\ell\in\R^2$, the far-field pattern $u_\ell(x;d,k)$ corresponding
to $\Om_\ell$ satisfies the relation
\be\label{TI-1}
u^\infty_{\ell}(\hat{x};d,k)=e^{ik\ell\cdot(d-\hat{x})}u^\infty(\hat{x};d,k),\;\;\hat{x}\in\Sp^1
\en
for any $\ell\in\R^2$. Therefore it is impossible to reconstruct the location of the obstacle $\Om$ from the
phaseless far-field pattern $|u^\infty(\hat{x};d,k)|$ with one plane wave as the incident field (i.e., with one incident
direction $d\in\Sp^1$). It was further proved in \cite{KR97} that this ambiguity cannot be remedied by using the phaseless
far-field pattern $|u^\infty(\hat{x};d,k)|$ with finitely many different wave numbers $k$ or different incident directions $d$.
Regularized Newton and Landweber iteration methods were also discussed in \cite{KR97} for the reconstruction of the
shape of the obstacle $\Om$ from the modulus of the far-field pattern with one plane wave as the incident field.
To reduce the computational cost, the nonlinear integral equation method was proposed
in \cite{Ivanyshyn07,IvanyshynKress2010} to reconstruct the shape of the obstacle $\Om$ from the modulus of
the far-field pattern with one plane wave as the incident field. Furthermore, in \cite{IvanyshynKress2010}
after the shape of the obstacle is reconstructed from the phaseless far-field data, an algorithm is proposed
for the localization of the obstacle by employing the relation (\ref{TI-1}) and using several full far-field
measurements at the backscattering direction.
Since the translation invariance relation (\ref{TI-1}) also holds in the case of sound-hard and impedance obstacles
(see \cite{LS04}), the algorithm proposed in \cite{IvanyshynKress2010} for recovering the location of the obstacle
works for these two cases as well.
In \cite{IvanyshynKress2011}, a nonlinear integral equation method was developed to recover its real-valued surface
impedance function from the modulus of the far-field pattern with one plane wave as the incident field
under the assumption that the scattering obstacle is known.
Recently, an inverse scattering scheme was proposed in \cite{LiLiu15} for recovering a polyhedral sound-soft or
sound-hard obstacle from a few high-frequency acoustic backscattering measurements. The method in \cite{LiLiu15}
first determines the exterior unit normal vector of each side/face of the obstacle by using the phaseless
backscattering far-field data corresponding to a few incident plane waves with suitably chosen incident directions
and then recovers a location point of the obstacle and distance of each side/face away from the location point
by making use of the far-field data with phases. In \cite{ACZ16}, a reconstruction algorithm was introduced
for reconstructing small perturbations of a known circle from the phaseless far-field data.
Recently, an iterative numerical algorithm was introduced in \cite{S16} to reconstruct the shape of a smooth and
strictly convex, sound-soft, obstacle from phaseless backscattering far-field data,
assuming sufficiently high frequency.
For plane wave incidence no uniqueness results are available for the general inverse obstacle scattering problems
with phaseless far-field data except for the case when the obstacle is a sound-soft ball centered
at the origin in \cite{LZ10}.

In contrast to the case with phaseless far-field data, the translation invariance property does not hold for the modulus
of the near-field (or the phaseless near-field), so many reconstruction algorithms have been developed to recover
the obstacle or the refractive index of the medium from the phaseless near-field data (or the phaseless total near-field
data), corresponding to plane wave incidence or point source incidence
(see, e.g. \cite{TGRS03,Li09,MDS92,MD93,Pan11,MOTL97} and the references quoted there).
Recently, for point source incidence uniqueness results have been established in \cite{K14} for the inverse scattering
problem of determining a nonnegative, smooth, real-valued potential with a compact support from the phaseless near-field data
corresponding to all incident point sources placed on a surface enclosing the compact support of the potential
for all wave numbers in a finite interval. This uniqueness result has been extended to the case of recovering
the smooth wave speed in the 3D Helmholtz equation in \cite{K17}.
Reconstruction procedures were introduced in \cite{KR16,N16,N15} for the inverse medium scattering problems
with phaseless near-field data. A direct imaging method was recently proposed in \cite{CH16,CH17} based on reverse time
migration for reconstructing extended obstacles by phaseless near-field electromagnetic and acoustic scattering data.

It should be pointed out that a continuation algorithm was proposed in \cite{BaoLiLv2012} to reconstruct the shape
of a perfectly reflecting periodic surface from the phaseless near-field data with one plane wave as the incident field
since the phaseless near-field is also invariant under translations in the non-periodic direction of the periodic surface.
Recently, a recursive linearization algorithm in frequencies was introduced in \cite{BaoZhang16} to recover the shape of
multi-scale sound-soft large rough surfaces from phaseless measurements of the scattered field generated by tapered waves
with multiple frequencies.

As discussed above, for plane wave incidence it is the translation invariance relation (\ref{TI-1}) which
makes it impossible to recover the location of the scattering obstacle from phaseless far-field data.
The purpose of this paper is to devise an approach to break the translation invariance property
of the phaseless far-field pattern. Precisely, instead of using one plane wave as the incident field,
we will make use of the following superposition of two plane waves as the incident field:
\be\label{IW}
u^i=u^i(x;d_1,d_2,k):=\exp({ikd_1\cdot x})+\exp({ikd_2\cdot x})
\en
where, for $j=1,2$, $d_j=(\cos\theta_j,\sin\theta_j)^T\in{\mathbb{S}}^1$
is the incident direction and $\theta_j\in[0,2\pi]$ is the incidence angle.
Then the scattered field $u^s$ will have the asymptotic behavior
\be\label{asymp-2}
u^s(x;d_1,d_2,k)=\frac{e^{ik|x|}}{\sqrt{|x|}}\left(u^\infty(\hat{x};d_1,d_2,k)
+O\Big(\frac{1}{|x|}\Big)\right),\qquad |x|\to\infty,
\en
uniformly for all observation directions $\hat{x}\in\Sp^1$.
We will establish conditions on the incident directions $d_j$ ($j=1,2$) under which the translation invariance
relation does not hold for $u^\infty(\hat{x};d_1,d_2,k)$ with all wave numbers $k$ in a finite interval.
Motivated by this, a recursive Newton iteration method in frequencies is then developed to recover both the location
and the shape of the obstacle $\Om$ simultaneously from multi-frequency phaseless far-field data.
Numerical examples are carried out to illustrate the validity of the approach.

The outline of this paper is as follows. In Section \ref{sec2}, properties of the far-field pattern
are studied for the scattering problem with superpositions of two plane waves as the incident fields,
and conditions are also obtained on the incident directions which ensure that
the translation invariance relation does not hold for all wave numbers in a finite interval.
In Section \ref{sec3}, a recursive Newton-type iteration method in frequencies is developed to solve
the inverse problems numerically with multi-frequency phaseless far-field data,
and numerical examples are carried out in Section \ref{sec4} to illustrate the validity of the new approach
and the effectiveness of the inversion algorithm.

\section{The direct problem with superpositions of plane waves as incident fields}\label{sec2}
\setcounter{equation}{0}

In this section we study properties of the far-field pattern $u^\infty(\hat{x};d_1,d_2,k)$
for the scattering problems (\ref{eq1})-(\ref{eq3}) and (\ref{eq4})-(\ref{eq7}) with the incident field
given by (\ref{IW}) which is a superposition of two plane waves.
In particular, we establish conditions on the incident directions $d_j$ ($j=1,2$) and the wave number $k$
which guarantee that the translation invariance relation does not hold for $u^\infty(\hat{x};d_1,d_2,k)$.
To this end, for any two different unit vectors $d_1,d_2\in\Sp^1$ (i.e., $d_1\neq d_2$) let $n_{d_1,d_2}\in\Sp^1$
denote the unit vector satisfying that $n_{d_1,d_2}\cdot d_1=n_{d_1,d_2}\cdot d_2$ (there are two such
unit vectors, so we simply choose one of them to be $n_{d_1,d_2}$).

We also need to introduce the following interior transmission problem (ITP): find a pair of functions $(u,v)$ such that
\ben\label{eq16}
\begin{cases}
\Delta u + k^2 n u = 0&\textrm{in}\;\;\Om\\
\Delta v + k^2 v = 0 &\textrm{in}\;\;\Om\\
u-v =0,\;\;{\pa u}/{\pa\nu}-\la{\pa v}/{\pa\nu}=0&\textrm{on}\;\;\G
\end{cases}
\enn
Note that $k > 0$ is called a transmission eigenvalue if the interior transmission problem (ITP)
has a nontrivial solution.
For a comprehensive discussion about the transmission eigenvalues, we refer to \cite{CH12}, where it is proved that
under certain conditions on $\la$ and the refractive index $n$, the set of real transmission eigenvalues of the
interior transmission problem (ITP) is discrete (see \cite[Theorems 3.6, 4.1 and 4.3]{CH12}).

\begin{theorem}\label{thm1}
Let $k>0$ and let $d\in\Sp^1$.
For the incident field $u^i=u^i(x;d,k)$ assume that $u^s(x;d,k)$ and $u^s_{\ell}(x;d,k)$ are the scattering
solutions to the problem $(\ref{eq1})-(\ref{eq3})$ with the same boundary condition (or the problem
$(\ref{eq4})-(\ref{eq7})$ with the same transmission constant $\la$), corresponding to the obstacle $\Om$ and
its shifted domain $\Om_\ell$, respectively. Assume further that $u^\infty(\hat{x};d,k)$ and
$u^\infty_{\ell}(\hat{x};d,k)$ are the far-field patterns of the scattered fields
$u^s(x;d,k)$ and $u^s_{\ell}(x;d,k)$, respectively. Then the translation invariance relation $(\ref{TI-1})$ holds
for any $\ell\in\R^2.$
\end{theorem}

\begin{proof}
The result is proved in \cite{KR97} for sound-soft obstacles and in \cite{LS04}
for sound-hard and impedance obstacles with constant impedance coefficients.
The case of penetrable obstacles can be proved similarly.
\end{proof}

\begin{theorem}\label{thm2}
Let $k>0$ and let $\ell\in\R^2$, $d_1,d_2\in\Sp^1$ with $d_1\neq d_2$.

Assume that $u^s(x;d_1,d_2,k)$ and $u^s_{\ell}(x;d_1,d_2,k)$ are the scattering solutions to
the problem $(\ref{eq1})-(\ref{eq3})$ with the same boundary condition (or the problem
$(\ref{eq4})-(\ref{eq7})$ with the same transmission constant $\la$), corresponding to the obstacle $\Om$
and its shifted domain $\Om_\ell$, respectively, and generated by the incident field $u^i=u^i(x;d_1,d_2,k)$.
Assume further that $u^\infty(\hat{x};d_1,d_2,k)$ and $u^\infty_{\ell}(\hat{x};d_1,d_2,k)$ are
the far-field patterns of the scattered fields $u^s(x;d_1,d_2,k)$ and $u^s_{\ell}(x;d_1,d_2,k)$, respectively.
For the case of penetrable obstacles we also assume that $k$ is not a transmission eigenvalue of the interior
transmission problem (ITP). Then
\be\label{TI-2}
\left|u^\infty_{\ell}(\hat{x};d_1,d_2,k)\right|=\left|u^\infty(\hat{x};d_1,d_2,k)\right|,\;\hat{x}\in\Sp^1
\en
for $\ell=\ell_n:=a n_{d_1,d_2}+[2\pi n/(k|d_1-d_2|^2)](d_1-d_2)$ with any $a\in\R$, where $n\in\Z$.
Further, except for $\ell_n$, there may exist at most a real constant $\tau$ with $0<\tau<2\pi$ such that
(\ref{TI-2}) holds for $\ell=\ell_{n\tau}:=a n_{d_1,d_2}+[(2\pi n+\tau)/(k|d_1-d_2|^2)](d_1-d_2)$
with any $a\in\R$ and $n\in\Z$.
\end{theorem}

\begin{proof}
We only consider the case of penetrable obstacles. The other cases can be proved similarly.

For $j=1,2$ let $u^\infty(\hat{x};d_j,k)$ and $u^\infty_{\ell}(\hat{x};d_j,k)$ be the far-field patterns of
the scattered fields $u^s(x;d_j,k)$ and $u^s_{\ell}(x;d_j,k)$ corresponding to the incident field $u^i=u^i(x;d_j,k)$
and the obstacles $\Om$ and $\Om_\ell$, respectively.
Since $u^i(x;d_1,d_2,k)=u^i(x;d_1,k)+u^i(x;d_2,k)$, we have by the conclusion (i) that for $\ell\in\R^2$
$$
u^\infty_{\ell}(\hat{x};d_1,d_2,k)=u^\infty_{\ell}(\hat{x};d_1,k)+u^\infty_{\ell}(\hat{x};d_2,k)
=e^{ik\ell\cdot(d_1-\hat{x})}u^\infty(\hat{x};d_1,k)+e^{ik\ell\cdot(d_2-\hat{x})}u^\infty(\hat{x};d_2,k).
$$
By a direct calculation it is easy to see that (\ref{TI-2}) holds if and only if $\ell$ satisfies the equation
\be\label{eq15}
\Rt\left(u^\infty(\hat{x};d_1,k)\ov{u^\infty(\hat{x};d_2,k)}\right)
=\Rt\left(e^{ik\ell\cdot(d_1-d_2)}u^\infty(\hat{x};d_1,k)\ov{u^\infty(\hat{x};d_2,k)}\right),\;\;
\hat{x}\in\Sp^1.
\en

We claim that $u^\infty(\hat{x};d_1,k)\ov{u^\infty(\hat{x};d_2,k)}\not\equiv 0,\;\;\hat{x}\in\Sp^1.$
Suppose this is not true. Then we would have
\be\label{eq18}
u^\infty(\hat{x};d_1,k)\ov{u^\infty(\hat{x};d_2,k)}\equiv 0,\;\;\hat{x}\in\Sp^1.
\en
We distinguish between the following two cases.

{\bf Case 1.} $u^\infty(\hat{x};d_1,k)=0$ for $\hat{x}\in\Sp^1$.

By Rellich's lemma \cite[Theorem 2.14]{CK13}, we have $u^s(x;d_1,k)=0$ in $\R^2\ba\ol{\Om}$.
Let $u(x;d_1,k)=u^i(x;d_1,k)+u^s(x;d_1,k)$ be the total field. Then $(u(x;d_1,k)|_\Om,u^i(x;d_1,k)|_\Om)$
satisfy the interior transmission problem (ITP). Since $k$ is not a transmission eigenvalue of the interior
transmission problem (ITP), we have $u^i(x;d_1,k)|_\Om=0$, which leads to the fact that $u^i(x;d_1,k)\equiv0$
in $\R^2$. This is a contradiction since $u^i(x;d_1,k)\neq0$ in $\R^2$.

{\bf Case 2.} $u^\infty(\hat{x}_0;d_1,k)\neq0$ for some point $\hat{x}_0\in\Sp^1$.

Since $u^\infty(\hat{x};d,k)$ is an analytic function of $\hat{x}\in\Sp^1$ for any $d\in\Sp^1$,
we have $u^\infty(\hat{x};d_1,k)\neq0$ in a neighborhood of $\hat{x}_0$.
Then, by (\ref{eq18}) we have $u^\infty(\hat{x};d_2,k)=0$ in a neighborhood of $\hat{x}_0$,
which, together with the analyticity in $\hat{x}$ of $u^\infty(\hat{x};d_2,k)$, implies that
$u^\infty(\hat{x};d_2,k)=0$ for all $\hat{x}\in\Sp^1$.
Arguing similarly as in Case 1, we have $u^i(x;d_2,k)\equiv0$ in $\R^2$, contradicting to
the fact that $u^i(x;d_2,k)\neq0$ in $\R^2$. Thus, the claim has been proved.

Now, by (\ref{eq15}) it is easy to see that (\ref{eq15}) or equivalently (\ref{TI-2}) always holds
if $\ell$ satisfies the condition
\be\label{TI-C1}
k\ell\cdot(d_1-d_2)=2\pi n,\qquad n\in\Z.
\en
Noting that $n_{d_1,d_2}\cdot d_1=n_{d_1,d_2}\cdot d_2$, we find that
the condition (\ref{TI-C1}) is satisfied if and only if
$\ell=a n_{d_1,d_2}+[2\pi n/(k|d_1-d_2|^2)](d_1-d_2)$ with any $n\in\Z$ and $a\in\R^2$.
In addition, assume that $v(\hat{x}_0):=u^\infty(\hat{x}_0;d_1,k)\ov{u^\infty(\hat{x}_0;d_2,k)}\not=0$
for some $\hat{x}_0\in\Sp^1$ and write $v(\hat{x}_0)=|v(\hat{x}_0)|\exp(i\theta(\hat{x}_0))$
with $0<\theta(\hat{x}_0)<2\pi$ and $\theta(\hat{x}_0)\not=\pi$.
Then (\ref{eq15}) is equivalent to the equation
\be\label{TI-C1a}
\cos[k\ell\cdot(d_1-d_2)+\theta(\hat{x}_0)]-\cos[\theta(\hat{x}_0)]=0.
\en
By (\ref{TI-C1a}) we know that, except for the above $\ell$ satisfying the condition (\ref{TI-C1}),
(\ref{eq15}) or equivalently (\ref{TI-2}) also holds for $\ell$ satisfying the condition
\be\label{TI-C1+}
k\ell\cdot(d_1-d_2)=2\pi n-2\theta(\hat{x}_0),\qquad n\in\Z,
\en
or equivalently for
\be\label{TI-C1b}
\ell=a n_{d_1,d_2}+\frac{2\pi n-2\theta(\hat{x}_0)}{k|d_1-d_2|^2}(d_1-d_2),\;\; n\in\Z,\;\; a\in\R^2.
\en
Thus, except for $\ell=\ell_n$, there may exist at most a real constant $\tau$ with $0<\tau<2\pi$ such that
(\ref{eq15}) or equivalently (\ref{TI-2}) holds for $\ell=a n_{d_1,d_2}+[(2\pi n+\tau)/(k|d_1-d_2|^2)](d_1-d_2)$
with any $a\in\R$ and $n\in\Z$. In fact, by (\ref{TI-C1b}) we have
\ben
\tau=\begin{cases}
2\pi-2\theta(\hat{x}_0)&\mbox{for}\;\;0<\theta(\hat{x}_0)<\pi,\\
4\pi-2\theta(\hat{x}_0)&\mbox{for}\;\;\pi<\theta(\hat{x}_0)<2\pi.
\end{cases}
\enn
The proof is thus complete.
\end{proof}

\begin{remark}\label{r0} {\rm
From the proof of Theorem \ref{thm2} we know that, if there exists a real constant $\tau$ with $0<\tau<2\pi$
such that (\ref{TI-2}) holds for $\ell=\ell_{n\tau}$ with $n\in\Z$ then the phase $\theta(\hat{x})$ of $u^\infty(\hat{x};d_1,k)\ov{u^\infty(\hat{x};d_2,k)}$ is equal to $\theta(\hat{x}_0)$ for all $\hat{x}\in\Sp^1$,
where $\theta(\hat{x}_0)$ is defined as in the proof of Theorem \ref{thm2} and related to $\tau$.
For the case when the obstacle is a circle, the far-field patterns $u^\infty(\hat{x};d_j,k)$, $j=1,2$, are given
explicitly, which can be used to prove that there is no $\tau\in(0,2\pi)$ such that $\theta(\hat{x})$
is a constant for all $\hat{x}\in\Sp^1$. This means that there is no $\tau\in(0,2\pi)$ such that (\ref{TI-2}) holds
for $\ell=\ell_{n\tau}$, that is, (\ref{TI-2}) only holds for $\ell=\ell_n$ with $n\in\Z$.
However, for general obstacles we are not able to prove this yet though we think it is true.
}
\end{remark}

\begin{remark}\label{r1} {\rm
By Theorem \ref{thm1} the translation invariance relation (\ref{TI-1}) holds for any $\ell\in\R^2$
if only one plane wave is used as the incident field, so the reconstructed obstacle can be translated in any direction,
depending on the initial guess of the iterative reconstruction algorithm, as shown in Figure \ref{fig1} below.
Theorem \ref{thm2} indicates that, by using a superposition of two plane waves with different
directions $d_1,d_2\in\Sp^1$ to be the incident field, the translation invariance relation (\ref{TI-1})
is restricted to hold only for $\ell=a n_{d_1,d_2}$ with any $a\in\R$. Thus, the reconstructed obstacle can only be
translated along a countably infinite number of straight lines which are parallel to the straight line $a n_{d_1,d_2}$,
as seen in Figures \ref{fig2} and \ref{fig3} below. Note that the two incident directions $d_1,d_2$ are symmetric
with respect to the straight line $a n_{d_1,d_2}$.
}
\end{remark}

As a corollary of Theorem \ref{thm2} we have the following result.

\begin{corollary}\label{cor1}
Let $k>0$, $\ell\in\R^2$ and let $d_{1l},d_{2l}\in\Sp^1$ with $d_{1l}\neq d_{2l}$, $l=1,2,$ and
$n_{d_{11},d_{21}}\nparallel n_{d_{12},d_{22}}$.

Assume that $u^s(x;d_{1l},d_{2l},k)$ and $u^s_{\ell}(x;d_{1l},d_{2l},k)$ are the scattering solutions to
the problem $(\ref{eq1})-(\ref{eq3})$ with the same boundary condition (or the problem
$(\ref{eq4})-(\ref{eq7})$ with the same transmission constant $\la$), corresponding to the obstacle $\Om$
and its shifted domain $\Om_\ell$, respectively, and generated by the incident field
$u^i=u^i(x;d_{1l},d_{2l},k)$, $l=1,2$. Assume further that $u^\infty(\hat{x};d_{1l},d_{2l},k)$ and
$u^\infty_{\ell}(\hat{x};d_{1l},d_{2l},k)$ are the far-field patterns of the scattered fields
$u^s(x;d_{1l},d_{2l},k)$ and $u^s_{\ell}(x;d_{1l},d_{2l},k)$, respectively.
For the case of penetrable obstacles we also assume that $k$ is not a transmission eigenvalue of the interior
transmission problem (ITP). Then we have
$$
|u^\infty(\hat{x};d_{1l},d_{2l},k)|=|u^\infty_{\ell}(\hat{x};d_{1l},d_{2l},k)|,\;\;\hat{x}\in\Sp^1,\;\;
l=1,2,
$$
only for $\ell$ being at the countably infinite number of cross points between the straight lines
$\ell_n=a n_{d_{11},d_{21}}+[2\pi n/(k|d_{11}-d_{21}|^2)](d_{11}-d_{21})$ with all $a\in\R^2$,
$\ell_m=b n_{d_{12},d_{22}}+[2\pi m/(k|d_{12}-d_{22}|^2)](d_{12}-d_{22})$ with all $b\in\R^2$,
$\ell_{n\tau_1}=a n_{d_{11},d_{21}}+[(2\pi n+\tau_1)/(k|d_{11}-d_{21}|^2)](d_{11}-d_{21})$ with all $a\in\R^2$
and $\ell_{m\tau_2}=a n_{d_{11},d_{21}}+[(2\pi m+\tau_2)/(k|d_{11}-d_{21}|^2)](d_{11}-d_{21})$ with all $a\in\R^2$,
where $n,m\in\Z$ and $\tau_j\in(0,2\pi)$ with $j=1,2$.
\end{corollary}

\begin{remark}\label{r2} {\rm
Corollary \ref{cor1} indicates that, by using superpositions of two plane waves with at least
two different sets of directions as the incident fields in conjunction with all wave numbers in a finite interval,
the translation invariance property of the modulus
of the far-field pattern can be broken down. Thus it is expected that both the location and the shape of the
obstacle $\Om$ can be numerically determined from phaseless far-field data corresponding to such incident waves,
as shown in Section \ref{sec4}.
However, we can not prove this theoretically so far. In \cite{KR97}, it has already been pointed out that
it is a very difficult problem to obtain an analogue for Schiffer's uniqueness result.
The proof of this result relies heavily on the fact that, by Rellich's lemma, the far-field pattern $u^\infty$
uniquely determines the scattered wave $u^s$. But a corresponding result is not available for the modulus
of the far-field pattern.
}
\end{remark}

Motivated by Corollary \ref{cor1} and Remark \ref{r2} we consider the following inverse problem.

{\bf Inverse Problem (IP).} Given the physical property of the obstacle $\Om$
(and the refractive index $n$ if $\Om$ is a penetrable obstacle) and
the incident fields $u^i=u^i(x;d_{1l},d_{2l},k_m)$, $l=1,2,\ldots,n_d$ with $n_d\geq2$, $m=1,2,\cdots,N$,
where $k_1<k_2<\cdots<k_N$ and $d_{1l},d_{2l}\in\Sp^1$ with $d_{1l}\neq d_{2l}$, $l=1,2,\ldots,n_d$,
and satisfy that there are at least two sets of unit vectors $\{d_{1i},d_{2i}\}$ and $\{d_{1j},d_{2j}\}$ with
$i\neq j$ and $n_{d_{1i},d_{2i}}\nparallel n_{d_{1j},d_{2j}}$,
to reconstruct both the location and the shape of the obstacle $\Om$ (and the constant $\la$ if
$\Om$ is a penetrable obstacle) from the corresponding phaseless far-field data
$|u^\infty(\hat{x};d_{1l},d_{2l},k_m)|,\;\;\hat{x}\in\Sp^1,\;\;l=1,2,\ldots,n_d$, $m=1,2,\cdots,N$.

Based on Corollary \ref{cor1} and Remark \ref{r2}, in the next section we will develop a
recursive Newton-type iterative algorithm in frequencies to numerically reconstruct both the location and
the shape of the obstacle $\Om$ from the phaseless far-field data, corresponding to the incident waves given
in the inverse problem (IP), and then carry out numerical experiments to illustrate the applicability
of the inversion algorithm in Section \ref{sec4}.

We first need to introduce the following notations.
For the case of an impenetrable obstacle $\Om$ we introduce the far-field operator $F_{I,d_1,d_2,k}$
which maps the boundary $\G$ to the corresponding phaseless far-field data:
\be\label{eq9}
(F_{I,d_1,d_2,k}[\G])(\hat{x})=|u^\infty(\hat{x};d_1,d_2,k)|^2,
\en
where the subscript $I=D,N$ stands for the Dirichlet, Neumann boundary condition on $\G$, respectively,
and $d_1,d_2$ denote the incident directions of the incident wave $u^i=u^i(x;d_1,d_2,k).$
Similarly, for the case of a penetrable obstacle $\Om$, we introduce the far-field operator $F_{T,d_1,d_2,k}$
which maps the boundary $\G$ and the constant $\la$ to the corresponding phaseless far-field data:
\be\label{eq10}
(F_{T,d_1,d_2,k}[\G,\la])(\hat{x})=|u^\infty(\hat{x};d_1,d_2,k)|^2.
\en

The Newton-type method consists in solving the non-linear and ill-posed equation (\ref{eq9}) for the unknown $\G$
(or (\ref{eq10}) for the unknowns $\G$ and $\la$). To this end, we need to investigate the Frechet differentiability
of the far-field operators.
Let $h\in C^1(\G)$ be a small perturbation of $\G$ and define $\G_{h}:=\{y\in\R^2|y=x+{h}(x),x\in\G\}$.
Then $F_{I,d_1,d_2,k}$ ($I=D,N$) is called Frechet differentiable at $\G$ if there exists a linear bounded
operator $F'_{I,d_1,d_2,k}:C^1(\G)\rightarrow L^2(\Sp^1)$ such that for $h\in C^1(\G)$ we have
\ben
\|F_{I,d_1,d_2,k}[\G_{h}]-F_{I,d_1,d_2,k}[\G]-F'_{I,d_1,d_2,k}[\G;{h}]\|_{L^2(\Sp^1)}
=o(\|h\|_{C^1(\G)})
\enn
Similarly, for $\la>0$, $F_{T,d_1,d_2,k}$ is called Frechet differentiable at $(\G,\la)$ if there exists a
linear bounded operator $F'_{T,d_1,d_2,k}:C^1(\G)\times\R\rightarrow L^2(\Sp^1)$ such that
for $(h,\triangle\la)\in C^1(\G)\times\R$ we have
\ben
\|F_{T,d_1,d_2,k}[\G_h,\la+\triangle\la]-F_{T,d_1,d_2,k}[\G,\la]
-F'_{T,d_1,d_2,k}[\G,\la;h,\triangle\la]\|_{L^2(\Sp^1)}
=o(\|h\|_{C^1(\G)}+|\triangle\la|)
\enn

The following theorem characterizes the Frechet derivatives of the far-field operators.

\begin{theorem}
Assume that $\G$ is a $C^2-$smooth boundary and  the incident field is given by
$u^i=u^i(x;d_1,d_2,k)$ with $d_1, d_2\in\Sp^1$ and $k>0$.
\begin{enumerate}[(i)]
\item Let $u=u^i+u^s$, where $u^s\in H^1_{loc}(\R^2\ba\ov{\Om})$ solves the scattering problem (\ref{eq1})-(\ref{eq3})
with the Dirichlet boundary condition on $\G$ and the boundary data $f=-u^i|_\G$.
Then $F_{D,d_1,d_2,k}$ is Frechet differentiable at $\G$ with the Frechet derivative given by
$F'_{D,d_1,d_2,k}[\G;h]=2 Re[\ov{u^\infty}(u')^\infty],\;h\in C^1(\G)$,
where $(u')^\infty$ is the far-field pattern of $u'\in H^1_{loc}(\R^2\ba\ov{\Om})$
which solves the scattering problem (\ref{eq1})-(\ref{eq3}) with the Dirichlet boundary condition on $\G$
and the boundary data $f=-h_\nu({\pa u}/{\pa\nu})|_\G$.

\item Let $u=u^i+u^s$, where $u^s\in H^1_{loc}(\R^2\ba\ov{\Om})$ solves the scattering problem (\ref{eq1})-(\ref{eq3})
with the Neumann boundary condition on $\G$ and the boundary data $f=-({\pa u^i}/{\pa\nu})|_\G$.
Then $F_{N,d_1,d_2,k}$ is Frechet differentiable at $\G$ with the Frechet derivative given by
$F'_{N,d_1,d_2,k}[\G;h]=2\Rt[\ov{u^\infty}(u')^\infty],\;h\in C^1(\G)$,
where $(u')^\infty$ is the far-field pattern of $u'\in H^1_{loc}(\R^2\ba\ov{\Om})$
which solves the scattering problem (\ref{eq1})-(\ref{eq3}) with the Neumann boundary condition on $\G$ and
the boundary data $f=k^2 h_\nu u|_\G+{\rm Div}_\G[h_\nu(\nabla u)_t]$.

\item  Let $u=u^i+u^s$, where $(u^s|_{\R^2\ba\ov{\Om}},u|_\Om)\in H^1_{loc}(\R^2\ba\ov{\Om})\times H^1(\Om)$ solves
the transmission scattering problem (\ref{eq4})-(\ref{eq7}) with the boundary data
$(f_1,f_2)=(-u^i|_\G,-({\pa u^i}{\pa\nu})|_\G)$.
Then $F_{T,d_1,d_2,k}$ is Frechet differentiable at $(\G,\la)$ with the Frechet derivative given
by $F'_{T,d_1,d_2,k}[\G,\la;h,\triangle\la]=2\Rt[\ov{u^\infty}(u')^\infty],\;h\in C^1(\G),\;\triangle\la\in\R$,
where $(u')^\infty$ is the far-field pattern of $u'\in H^1_{loc}(\R^2\ba\ov{\Om})\cap H^1(\Om)$
which solves the transmission scattering problem (\ref{eq4})-(\ref{eq7}) with the boundary data
$f_1=-h_\nu({\pa u_+}/{\pa\nu}-{\pa u_-}/{\pa\nu})|_\G$ and
$f_2=(k^2-\la k^2 n)h_\nu u|_\G+{\rm Div}_\G[h_\nu((\nabla u_+)_t-\la(\nabla u_-)_t)]
+({\triangle\la}/{\la})({\pa u_+}/{\pa\nu})|_\G$ on $\G$.
\end{enumerate}
where $h_\nu$ and $h_t$ denote the normal and tangential components of a vector field $h$.
\end{theorem}

\begin{proof}
(i) has been proved in \cite{KR97}.
Note that, since $\G\in C^2$, it is easily seen by the elliptic regularity estimates \cite{GT83} that
$u\in H^2_{loc}(\R^2\ba\ov{\Om})$ which guarantees that $u'\in H^1_{loc}(\R^2\ba\ov{\Om})$.

(ii) can be shown by using the product rule and Theorem 2.1 in \cite{H95}.

(iii) can also be proved similarly by using the product rule and following the idea of the proof of Theorem A.1
in \cite{ZZ13}.
\end{proof}

\section{The reconstruction algorithm}\label{sec3}
\setcounter{equation}{0}

In this section, we develop a Newton-type iteration algorithm for the inverse problem (IP)
for the case when $\Om$ is a penetrable obstacle. For other cases, the steps are similar.

Assume that the boundary $\G$ is a starlike curve which can be parameterized by $\g$ such that
\ben
\g(\theta)=(a_1,a_2)+r(\theta)(\cos\theta,\sin\theta)^T,\quad \theta\in[0,2\pi]
\enn
with its center at $(a_1,a_2)$. For the fixed wave number $k>0$, given the phaseless far-field data
$|u^\infty(\hat{x};d_{1l},d_{2l},k)|,\;d_{1l},d_{2l}\in\Sp^1,\;l=1,2,\ldots,n_d$
for the scattering problem (\ref{eq4})-(\ref{eq7}), where $n_d\geq2$ and the incident directions
$d_{1l},d_{2l}$ satisfy the conditions given in the inverse problem (IP),
(\ref{eq10}) can be rewritten as:
\be\label{eq11}
(F_{T,d_{1l},d_{2l},k}[\g,\la])(\hat{x})=|u^\infty(\hat{x};d_{1l},d_{2l},k)|^2,\quad l=1,2,\ldots,n_d.
\en
Assume that $\g^{app}=(a^{app}_1,a^{app}_2)+r^{app}(\theta)(\cos\theta,\sin\theta)^T$ and $\la^{app}$
are the approximations to $\g$ and $\la$, respectively.
Then the equation (\ref{eq11}) can be linearized at $(\g^{app},\la^{app})$ as follows:
\be\label{eq12}
(F_{T,d_{1l},d_{2l},k}[\g^{app},\la^{app}])(\hat{x})+(F'_{T,d_{1l},d_{2l},k}[\g^{app},\la^{app};
\triangle\g,\triangle\la])(\hat{x})\approx|u^\infty(\hat{x};d_{1l},d_{2l},k)|^2
\en
for $l=1,2,\ldots,n_d$, where $(\triangle\g,\triangle\la)$ are the updates to be determined.
Our Newton method consists in iterating the ill-posed equations (\ref{eq12}) with using
the Levenberg-Marquardt algorithm (see, e.g., \cite{H97,H99}).

In the numerical computation, we consider the noisy phaseless far-field data $|u^\infty_\delta|$
which satisfies that
\ben
\big\||u^\infty_\delta|^2-|u^\infty|^2\big\|_{L^2(\Sp^1)}\leq\delta\big\||u^\infty|^2\big\|_{L^2(\Sp^1)}
\enn
where $\delta>0$ is called the noise ratio. Further, $r^{app}$ has to be taken from a finite-dimensional
subspace $R_M\subset H^s(0,2\pi)$, $s\geq0$, where
\be\label{eq12+}
R_M:=\{r\in H^s(0,2\pi)\;|\;r(\theta)=\alpha_0+\sum^M_{l=1}[\alpha_l\cos(l\theta)+\alpha_{l+M}\sin(l\theta)],\;
\alpha_l\in\R\;\textrm{for}\;l =0,\ldots,2M\}
\en
with the norm
$$
\|r\|^2_{H^s(0,2\pi)}:=2\pi\alpha^2_0+\pi\sum^M_{l=1}[(1+l^2)^s(\alpha_l^2+\alpha^2_{l+M})]
$$
(see \cite[pp. 120]{H99}). Then, by using the strategy in \cite{H97,H99},
we seek the updates $(\triangle\g,\triangle\la)$, where
$\triangle\g=(\triangle a_1,\triangle a_2)+\triangle r(\theta)(\cos\theta,\sin\theta)$,
$\triangle a_1,\;\triangle a_2,\;\triangle\la\in\R,\;\triangle r\in R_M$
such that $(\triangle a_1,\;\triangle a_2,\;\triangle r,\;\triangle\la)$ solves the minimization problem:
\be\label{eq13}
&&\min_{\triangle a_1,\triangle a_2,\triangle r,\triangle\la}
\left\{\sum^{n_d}_{l=1}\big\|(F_{T,d_{1l},d_{2l},k}[\gamma^{app},\la^{app}])(\hat{x})
+(F'_{T,d_{1l},d_{2l},k}[\gamma^{app},\la^{app};\triangle\gamma,\triangle\la])(\hat{x})\right.\no\\
&&\qquad\qquad\qquad\left.-|u^\infty_\delta(\hat{x};d_{1l},d_{2l},k)|^2\big\|^2_{L^2({\mathbb{S}}^1)}
+\beta\left(\sum^2_{l=1}|\triangle a_l|^2+\|\triangle r\|^2_{H^s(0,2\pi)}+|\triangle\la|^2\right)\right\}
\en
where $\beta\in\R^+$ is chosen so that
\be\label{eq14}
&&\left(\sum^{n_d}_{l=1}\big\|(F_{T,d_{1l},d_{2l},k}[\gamma^{app},\la^{app}])(\hat{x})
+(F'_{T,d_{1l},d_{2l},k}[\gamma^{app},\la^{app};\triangle\gamma,\triangle\la])(\hat{x})
-|u^\infty_\delta(\hat{x};d_{1l},d_{2l},k)|^2\big\|^2_{L^2({\mathbb{S}}^1)}\right)^{1/2}\no\\
&&\qquad\qquad\qquad=\rho\left(\sum^{n_d}_{l=1}\big\|(F_{T,d_{1l},d_{2l},k}[\gamma^{app},\la^{app}])(\hat{x})
-|u^\infty_\delta(\hat{x};d_{1l},d_{2l},k)|^2\big\|^2_{L^2({\mathbb{S}}^1)}\right)^{1/2}
\en
with a given constant $\rho<1$. Here, we use the bisection algorithm to determine $\beta$ (see \cite{H99}).
Then the approximations $\g^{app}$ and $\la^{app}$ are updated by $\g^{app}+\triangle\g$ and
$\la^{app}+\triangle\la$, respectively.

Define the relative error by
\ben
\textrm{Err}_k:=\frac{1}{n_d}\sum^{n_d}_{l=1}\frac{\big\|(F_{T,d_{1l},d_{2l},k}[\g^{app},\la^{app}])(\hat{x})
-|u^\infty_\delta(\hat{x};d_{1l},d_{2l},k)|^2\big\|_{L^2({\mathbb{S}}^1)}}
{\big\||u^\infty_\delta(\hat{x};d_{1l},d_{2l},k)|^2\big\|_{L^2({\mathbb{S}}^1)}}
\enn
Then the iteration is stopped if ${\rm Err}_k<\tau\delta$, where $\tau>1$ is a given constant.

\begin{remark}\label{r3} {\rm
In the numerical examples carried out later, we use the Nystr\"{o}m method (see, e.g. \cite{CK13,HS98,K95})
to compute the synthetic data and the numerical solution in each iteration.
We use $|u^\infty_\delta(\hat{x_j};d_{1l},d_{2l},k)|$, $d_{1l},d_{2l}\in\Sp^1$,
$j=1, 2,\ldots,n_f$, $l=1,2,\ldots,n_d$ as the measured phaseless far-field data,
where $\hat{x}_j =(\cos(\theta_j), \sin(\theta_j))$, $\theta_j = 2\pi(j -1)/n_f$
and the norm $\|\cdot\|_{L^2(\Sp^1)}$ is approximated by
\ben
\|g\|^2_{L^2(\Sp^1)}\approx\frac{2\pi}{n_f}\sum^{n_f}_{j=1}|g(\hat{x}_j)|^2,\;\;\;g\in C(\Sp^1)
\enn
For the computation of the Frechet derivative of the far-field operator, we refer to \cite{H98,HS98,K93}.
}
\end{remark}

{
Motivated by Corollary \ref{cor1} and Remark \ref{r2} as well as \cite{BLLT15}, our reconstruction algorithm
makes use of multi-frequency phaseless far-field measurement data $|u^\infty_{\delta}(\hat{x_j};d_{1l},d_{2l},k_m)|^2$,
$j=1,2,\ldots,n_f$, $l=1,2,\ldots,n_d$, $m=1,2,\ldots,N$ with $d_{1l},d_{2l}\in\Sp^1$, $d_{1l}\neq d_{2l}$
and $k_1<k_2<\cdots<k_N$. We first choose $(\g^{app},\la^{app})$ as the initial guess of $(\g,\la)$ with $\la^{app}>0$
and solve (\ref{eq13}) for the update $(\triangle\g,\triangle\la)$ with the strategy (\ref{eq14}) with $k=k_1$
to get the reconstruction $(\g^{app}_1,\la^{app}_1):=(\g^{app}+\triangle\g,\la^{app}+\triangle\la)$ of $(\g,\la)$
at the wavenumber $k_1$. We then use the reconstruction $(\g^{app}_1,\la^{app}_1)$ for $k_1$ as the initial guess
of $(\g,\la)$ for the next wavenumber $k_2$ and solve (\ref{eq13}) for the update $(\triangle\g,\triangle\la)$ with
the strategy (\ref{eq14}) with $k=k_2$, $(\g^{app},\la^{app}):=(\g^{app}_1,\la^{app}_1)$ to get the reconstruction
$(\g^{app}_2,\la^{app}_2):=(\g^{app}+\triangle\g,\la^{app}+\triangle\la)$ of $(\g,\la)$ at the next wavenumber $k_2$.
Repeat this process until the reconstruction of $(\g,\la)$ is obtained at the highest wavenumber $k_N$.
Our recursive Newton iteration algorithm in frequencies for the inverse problem (IP) is based on the above process
and presented in Algorithm \ref{alg1} below in the case of a penetrable obstacle.
}

\begin{algorithm}\label{alg1}
Given the refractive index $n$ and the phaseless far-field data $|u^\infty_{\delta}(\hat{x}_j;d_{1l},d_{2l},k_m)|^2$,
$d_{1l},d_{2l}\in\Sp^1$, $j=1,2,\ldots,n_f$, $l=1,2,\ldots,n_d$, $m=1,2,\ldots,N$ for the scattering problem
(\ref{eq4})-(\ref{eq7}), where $k_1<k_2<\cdots<k_N$, $n_d\geq2$ and the incident directions
$d_{1l}$ and $d_{2l}$ satisfy the conditions given in the inverse problem (IP).
\begin{description}
\item Step 1. Choose initial guesses $(\g^{app},\la^{app})$ for $(\g,\la)$, where $\la^{app}>0$.
Set $m=0$ and go to Step 2.
\item Step 2. Set $m=m+1$. If $m>N$, then stop the iteration; otherwise, set $k=k_m$ and go to Step 3.
\item Step 3. If $Err_k<\tau\delta$, go to Step 2; otherwise, go to Step 4.
\item Step 4. Solve (\ref{eq13}) with the strategy (\ref{eq14}) to get the updates $(\triangle\g,\triangle\la)$.
Let $(\g^{app},\la^{app})$ be updated by $(\g^{app}+\triangle\g,\la^{app}+\triangle\la)$, respectively,
and go to Step 3.
\end{description}
\end{algorithm}

{
\begin{remark}\label{r4} {\rm
As discussed above, using superpositions of plane waves as the incident fields and multi-frequency phaseless
far-field measurements can recover both the shape and the location of the obstacle simultaneously.
However, in order to get a much better reconstruction of the obstacle we need to consider scattering data
produced at multi-frequencies (or wavenumbers) from different regions (see \cite{BLLT15}), as seen in the numerical
experiments in Section \ref{sec4}. At a low wavenumber, the scattered field is weak, so the nonlinear
equation (\ref{eq11}) becomes essentially linear. Algorithm \ref{alg1} first solves this nearly linear equation
at the lowest wavenumber to obtain low-frequency modes of the true obstacle. The approximation is then used to
linearize the nonlinear equation at the next higher wavenumber to produce a better approximation which contains
more modes of the true obstacle. This process is continued until a sufficiently high wavenumber where the dominant
modes of the obstacle are essentially recovered. The successive recovery is permitted by the Heisenberg uncertainty
principle: it is increasingly difficult to determine features of the obstacle as its size becomes decreasingly smaller
than a half of a wavelength.
}
\end{remark}
}

\section{Numerical examples}\label{sec4}

In this section, several numerical examples are carried out to demonstrate whether or not the location and
the shape of obstacles can be reconstructed simultaneously from phaseless far-field data, by using a superposition
of two plane waves as the incident field in the inverse problem (IP).
In all numerical examples, we make the following assumptions.
\begin{enumerate}[(1)]
\item For each example we use multi-frequency data with the wave numbers $k=0.5,1,3,5,7,9,11$ unless otherwise stated.

\item To generate the synthetic data, we use the integral equation method with the Nystr\"om algorithm
(see, e.g. \cite{CK13,HS98,K95}) to solve the direct problems and measure the full phaseless far-field data
with $128$ measurement points.
The noisy data $|u^\infty_{\delta}|^2$ is simulated by $|u^\infty_{\delta}|^2=|u^\infty|^2(1+\delta\zeta)$,
where $\zeta$ is a normally distributed random number in $[-1,1]$.

\item For the parameters in the algorithm, we choose $s=1.6$, as suggested in \cite{H99},
$M=25$, $\rho=0.8$ and $\tau=1.5$.

\item In each figure, we use the solid line '---' to represent the exact boundary curves of the obstacles.
Further, all the initial guesses of the boundaries of the obstacles are chosen to be circles.

\item The parametrization of the exact boundary curves we used are given in Table \ref{table2}.
\end{enumerate}
\begin{table}[h]
\centering
\begin{tabular}{ll}
\hline
Type & Parametrization\\
\hline
Circle & $r_0(\cos{t},\sin{t})+(c_1,c_2),\;t\in[0,2\pi]$\\
Apple shaped & $[({0.5+0.4\cos{t}+0.1\sin(2t)})/({1+0.7\cos{t}})](\cos{t},\sin{t}),\;t\in[0,2\pi]$\\
Kite shaped & $(\cos{t}+0.65\cos(2t)-0.65,1.5\sin{t}),\;t\in[0,2\pi]$\\
Rounded triangle & $(2+0.3\cos(3t))(\cos{t},\sin{t}),\;t\in[0,2\pi]$\\
\hline
\end{tabular}
\caption{Parametrization of the exact boundary curves}\label{table2}
\end{table}

\textbf{Example 1: Shape reconstruction of a sound-soft obstacle.}

We first consider the inverse problem of reconstructing a sound-soft, apple-shaped obstacle
from phaseless far-field data with $5\%$ noise, by using only one plane wave as the incident field
$u^i=u^i(x;d,k)$ with the direction $d=(1,0)$ and multiple frequencies.
As discussed in Remark \ref{r1}, we can only reconstruct the shape of the obstacle, and
the reconstructed obstacle can be translated into any location, depending on the initial guess of
our inversion algorithm. To verify this numerically, we choose three different initial guesses
of the exact boundary $\G$ which are circles with radius $r_0=0.5$ and three different centers
at $(-1.5,0)$, $(0.5,1)$ and $(1,-1)$, and the corresponding reconstructed curves
are denoted as "Reconstructed curve 1", "Reconstructed curve 2" and "Reconstructed curve 3",
respectively. In Figure \ref{fig1}, we present the initial curves and the reconstructed curves at $k=1,5,11.$
It is shown that the shape of the obstacle is reconstructed accurately.
But the reconstructed obstacle has three different locations depending on the choice of the initial guesses
even though multi-frequency data are used.

Next, we consider the same inverse problem as above, but with a superposition of two plane waves
as the incident field, that is, the incident field $u^i=u^i(x;d_1,d_2,k)$ with $d_1\not=d_2.$
Again, we use multi-frequency phaseless far-field data with $5\%$ noise.
We choose the same three initial guesses of the obstacle.
Figure \ref{fig2} presents the initial curves and the reconstructed curves
at $k=1,5,11$ with the two incident directions $d_1=(-1,0)$ and $d_2=(0,-1)$.
It can be seen that the shape of the obstacle is accurately reconstructed but its location is still not determined.
In this case, the reconstructed obstacle is translated only along the straight line
$\ell=a(\sqrt{2}/2,\sqrt{2}/2)$ with $a\in\R$ which is the symmetric line of the two incident directions
$d_1=(-1,0)$ and $d_2=(0,-1)$.
Figure \ref{fig3} presents the initial curves and the reconstructed curves at $k=1,5,11$
with the incident directions $\tilde{d_1}=(1,0)$, $\tilde{d_2}=(0,-1)$.
Similarly as in Figure \ref{fig2}, the shape of the obstacle is accurately reconstructed,
but its location is translated along the straight line $\ell=a(-\sqrt{2}/2,\sqrt{2}/2)$ with $a\in\R$
which is the symmetric line of the two incident directions $\tilde{d_1}=(1,0)$, $\tilde{d_2}=(0,-1)$.
This is consistent with the theoretical result in Theorem \ref{thm2}.

Finally, we consider the inverse problem (IP) for the same obstacle as above by using phaseless
far-field data with only one frequency. The wave number is chosen to be $k=3$.
Again, we use the same three initial guesses for the obstacle.
The phaseless far-field data are perturbed by $5\%$ noise and generated by
the incident field $u^i=u^i(x;d_{1l},d_{2l},k)$, $l=1,2,$ with two different sets of directions
$d_{11}=(1,0), d_{21}=(-1/2,\sqrt{3}/2)$ and $d_{12}=(1,0),d_{22}=(-1/2,-\sqrt{3}/2)$.
Figure \ref{fig10} presents the initial curves and the reconstructed curves.
From Figure \ref{fig10} it can be seen that the location of the obstacle
is still not determined. Therefore, in the following examples
we consider the inverse problem (IP) with multi-frequency measured data.

\begin{figure}[htbp]
  \centering
  \subfigure{\includegraphics[width=3in]{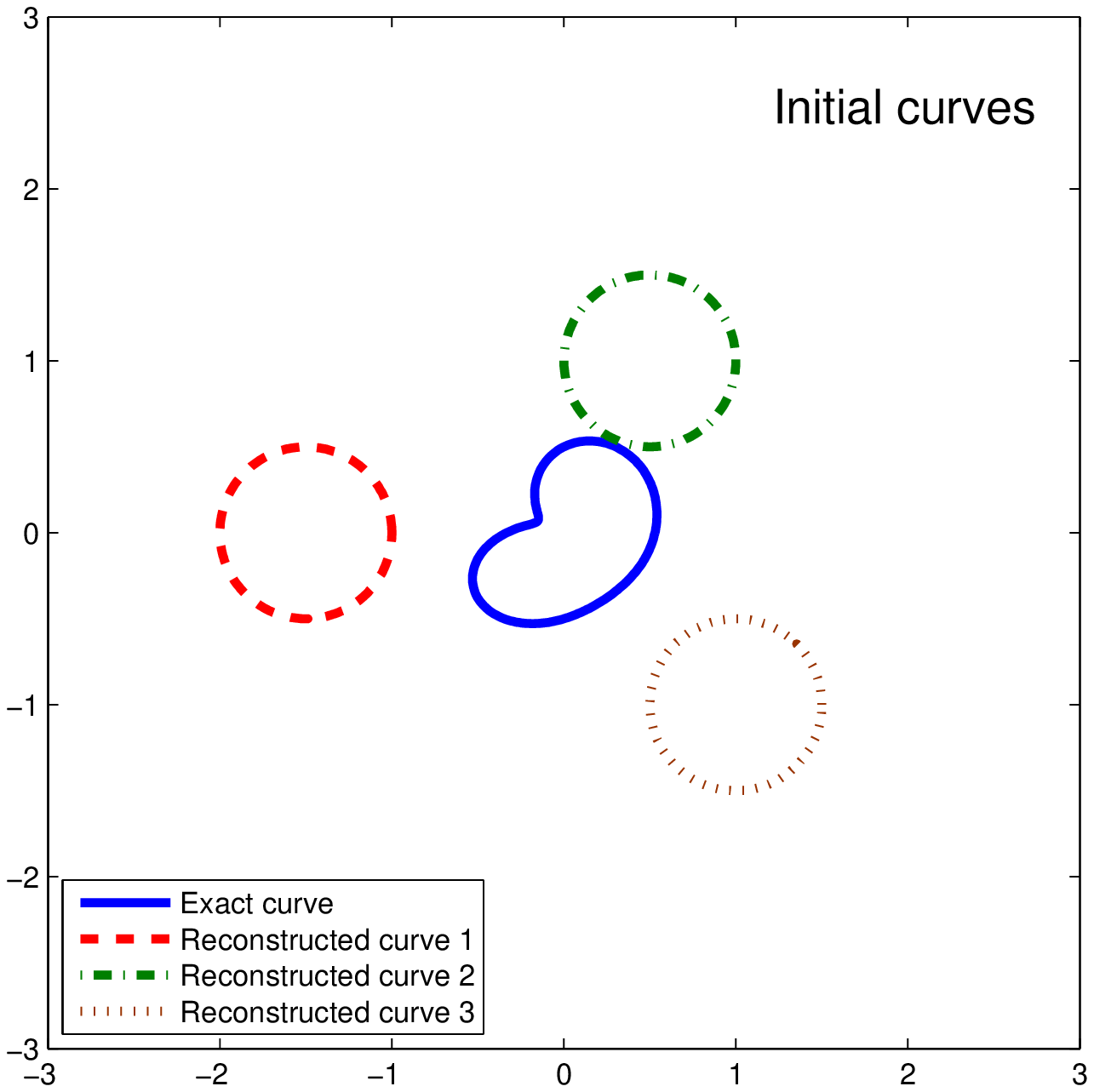}}
  \subfigure{\includegraphics[width=3in]{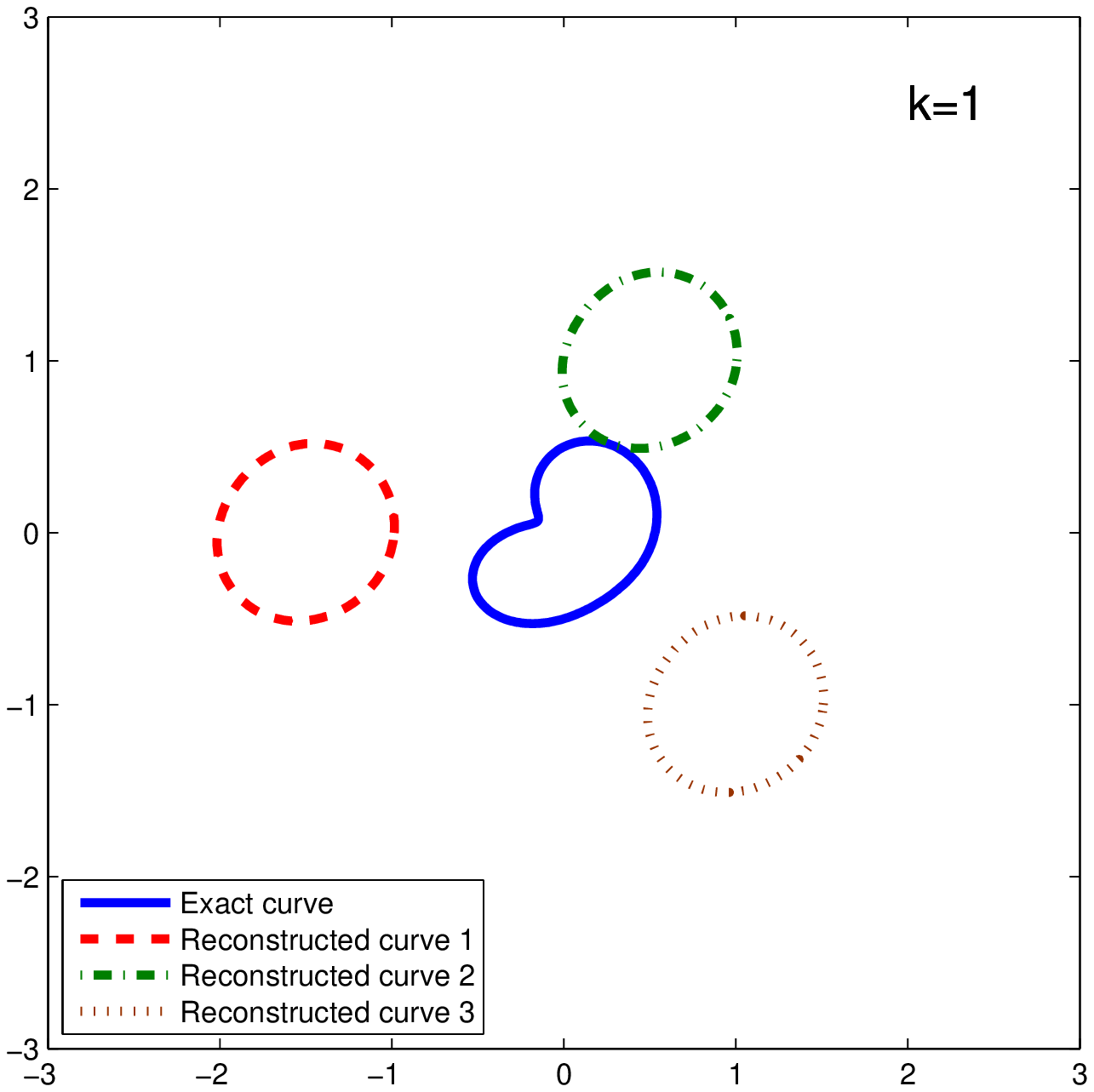}}
  \subfigure{\includegraphics[width=3in]{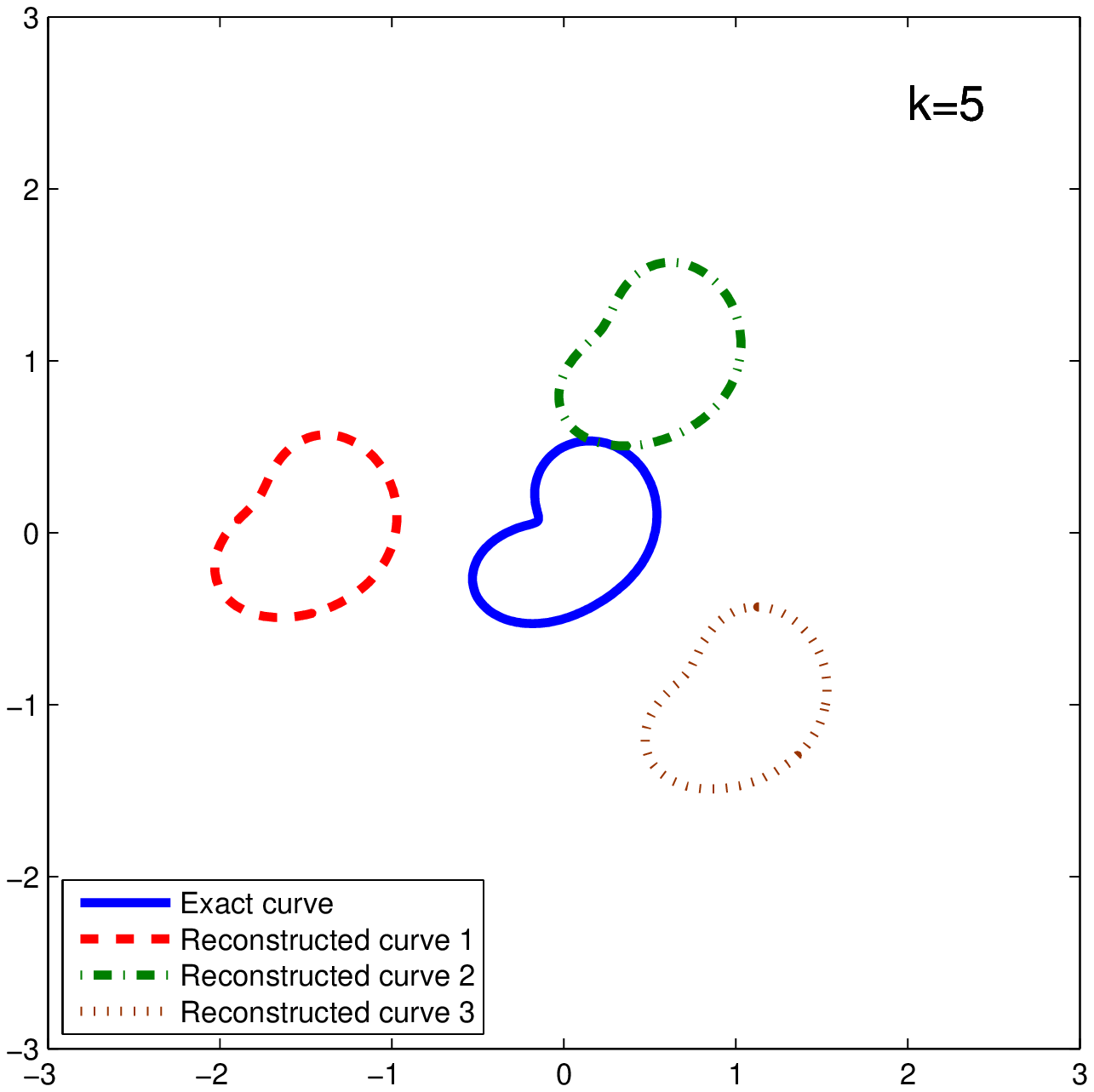}}
  \subfigure{\includegraphics[width=3in]{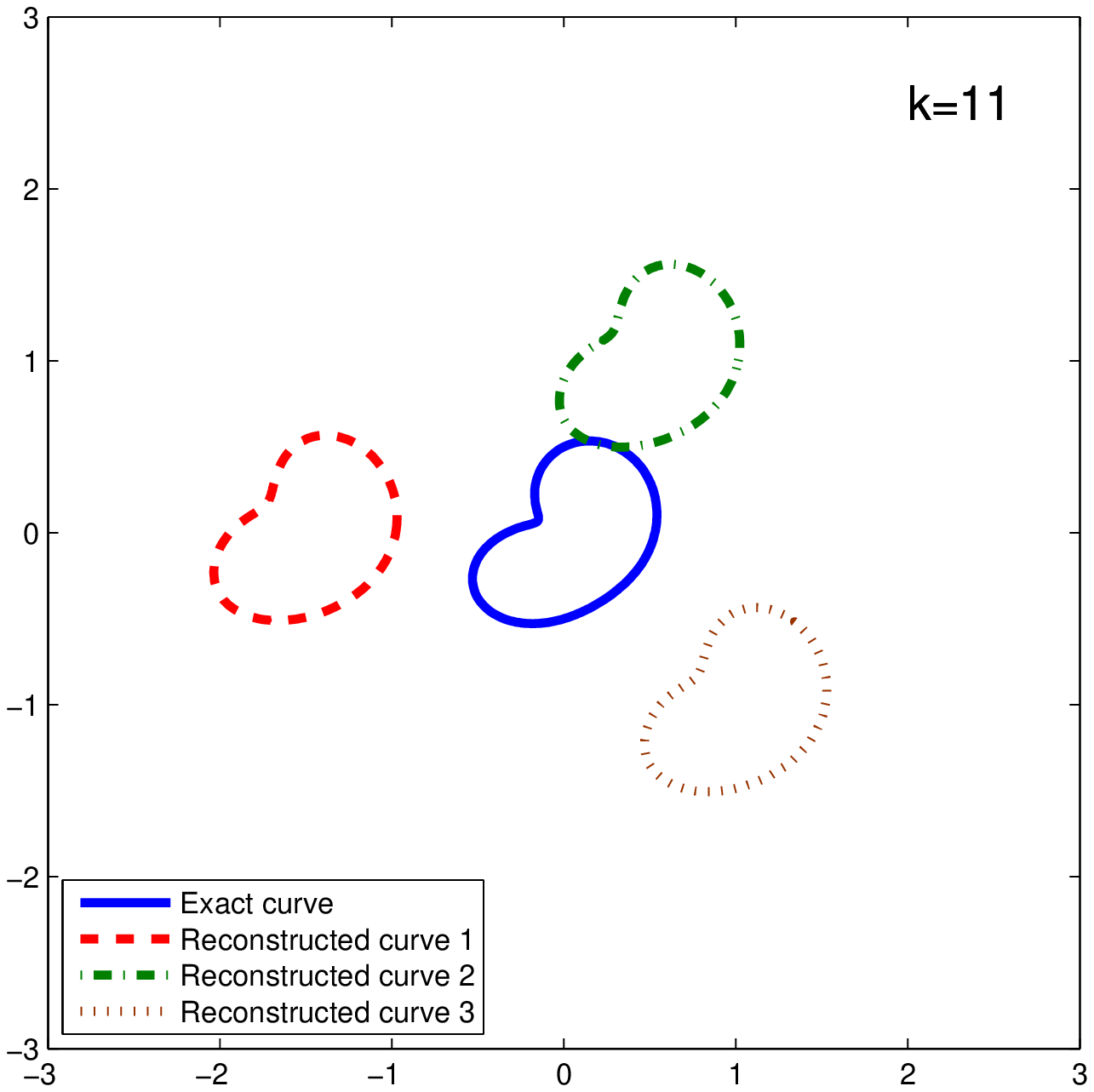}}
\caption{Shape reconstruction of a sound-soft, apple-shaped obstacle from the phaseless far-field data with $5\%$ noise,
corresponding to only one incident plane wave $u^i=u^i(x;d,k)$ with the direction $d=(1,0)$ and multiple frequencies.
The initial guesses and the reconstructed obstacles at $k=1,5,11$ are presented.
}\label{fig1}
\end{figure}

\begin{figure}[htbp]
  \centering
  \subfigure{\includegraphics[width=3in]{pic/example1/initial_curves.eps}}
  \subfigure{\includegraphics[width=3in]{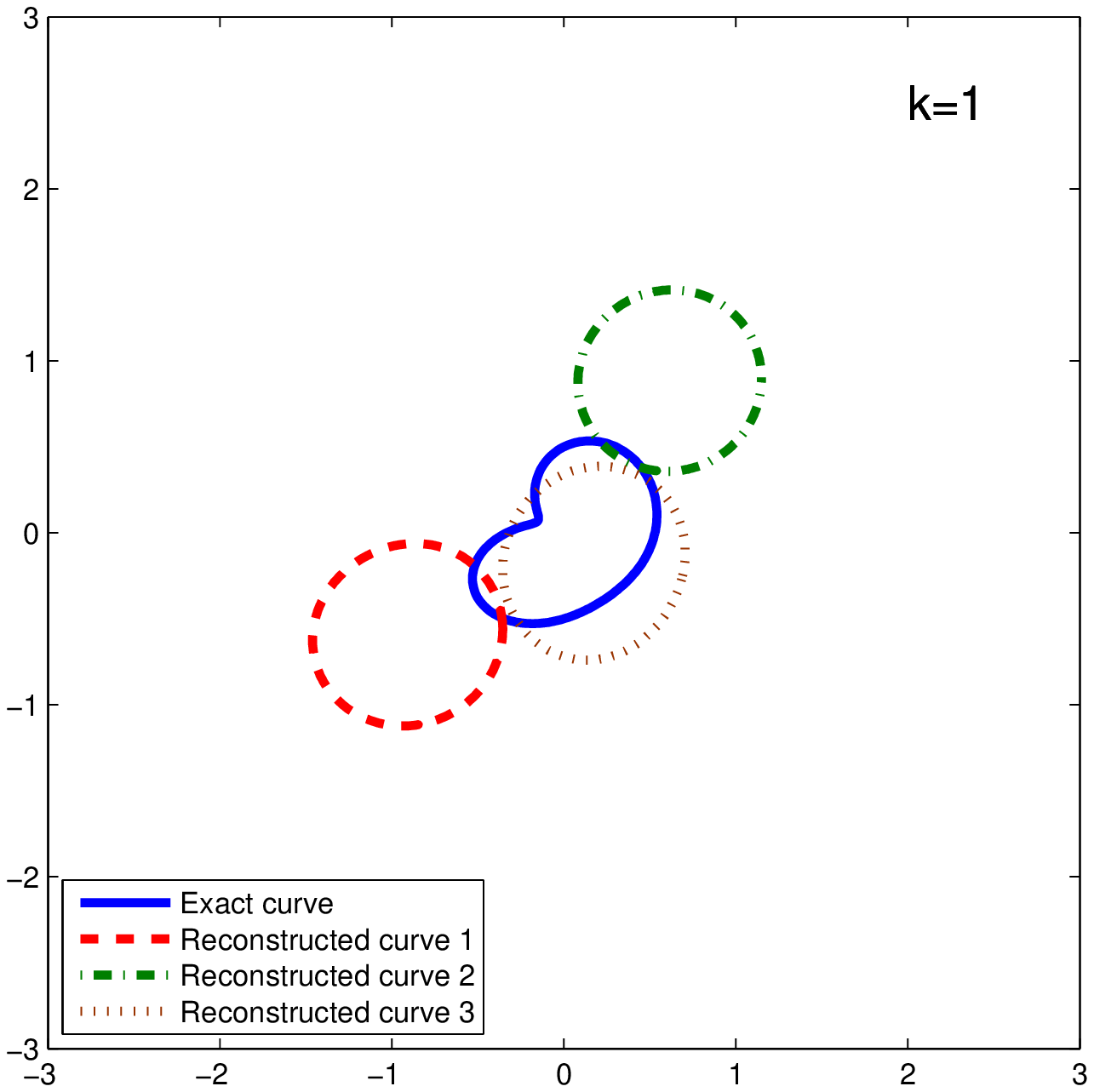}}
  \subfigure{\includegraphics[width=3in]{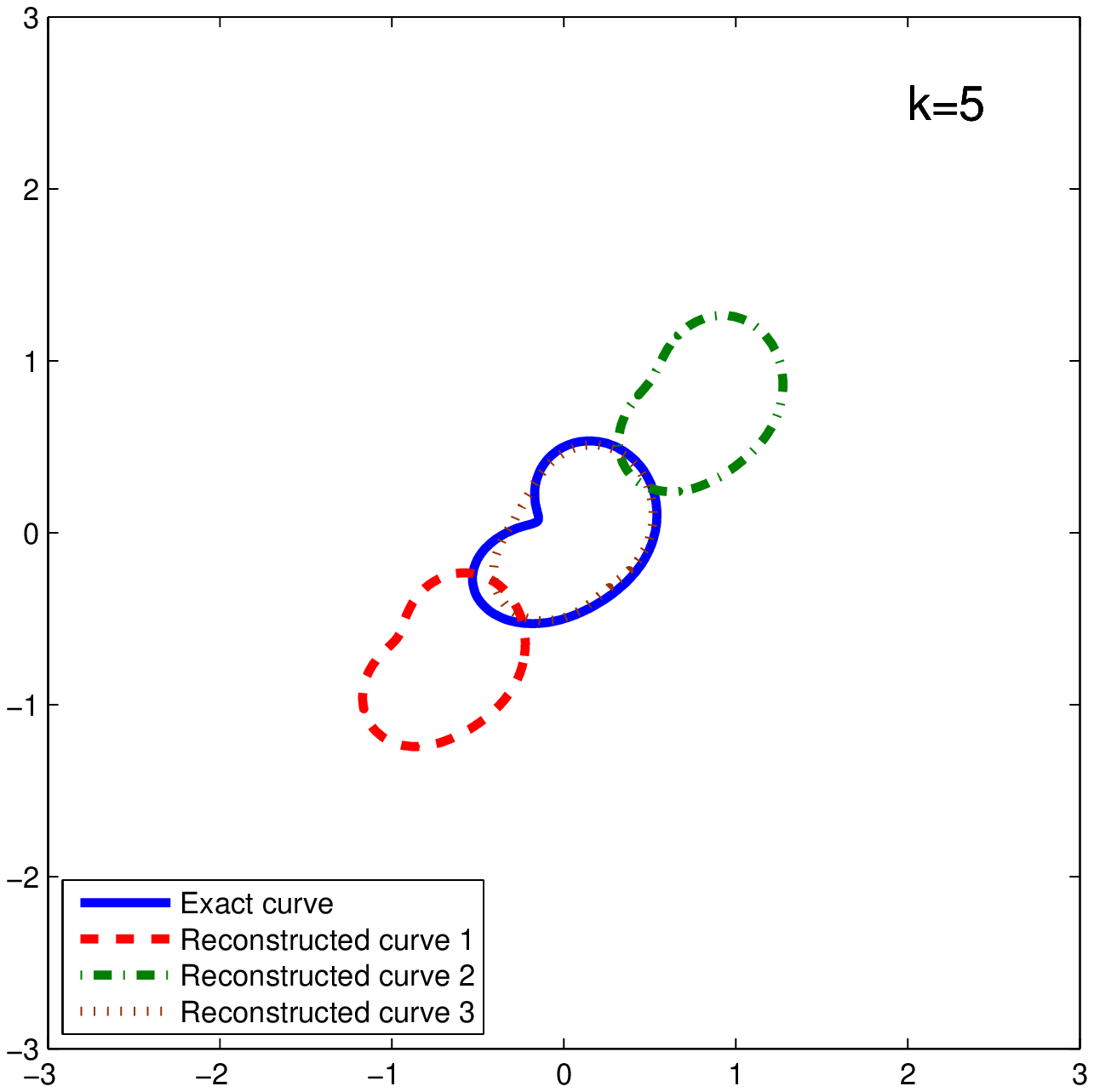}}
  \subfigure{\includegraphics[width=3in]{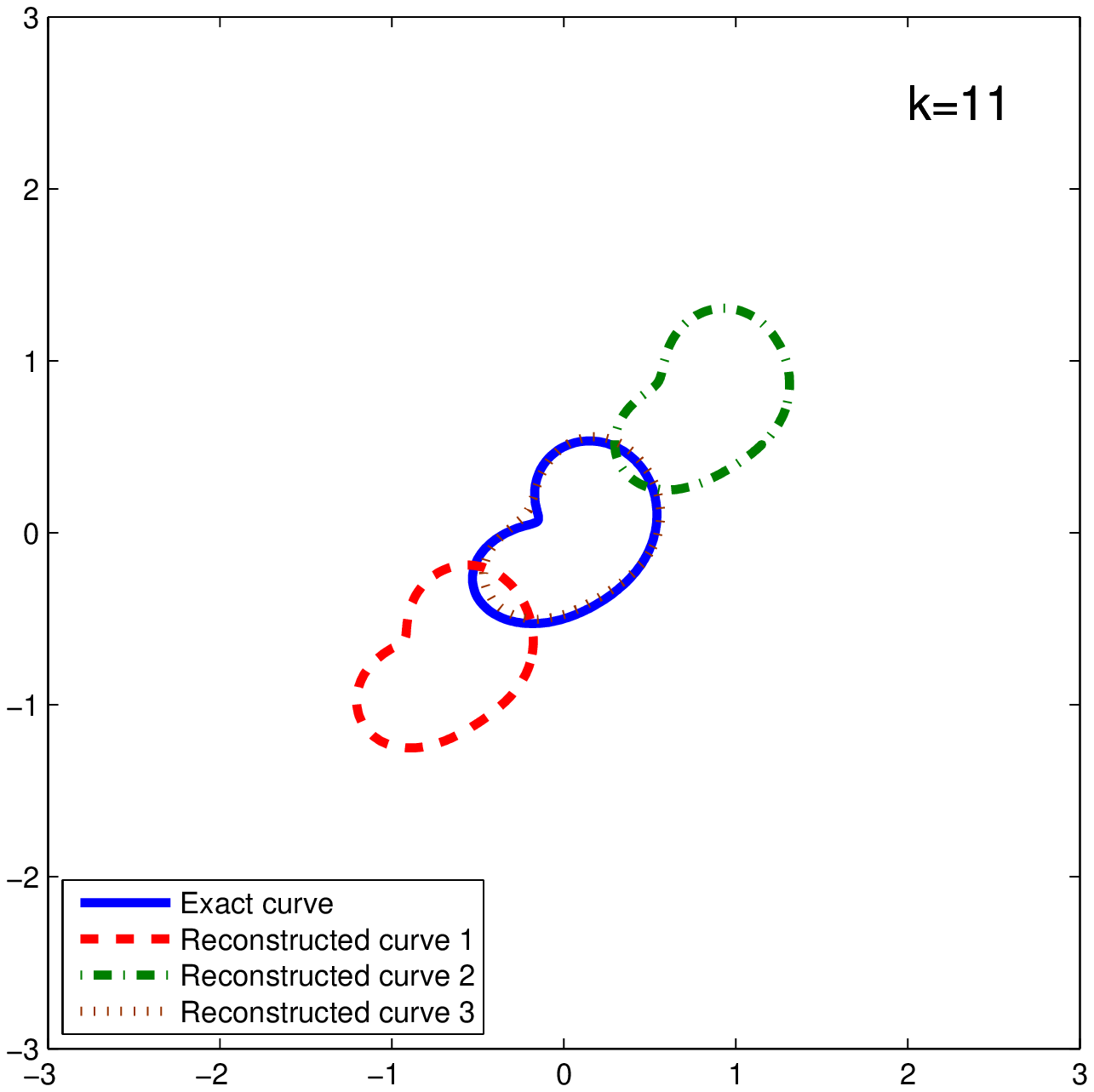}}
\caption{Shape reconstruction of a sound-soft, apple-shaped obstacle from the phaseless far-field data with $5\%$ noise,
corresponding to the superposition of two plane waves as the incident field $u^i=u^i(x;d_1,d_2,k)$
with multiple frequencies and two different directions $d_1=(-1,0)$ and $d_2=(0,-1)$.
The initial guesses and the reconstructed obstacles at $k=1,5,11$ are presented.
}\label{fig2}
\end{figure}

\begin{figure}[htbp]
  \centering
  \subfigure{\includegraphics[width=3in]{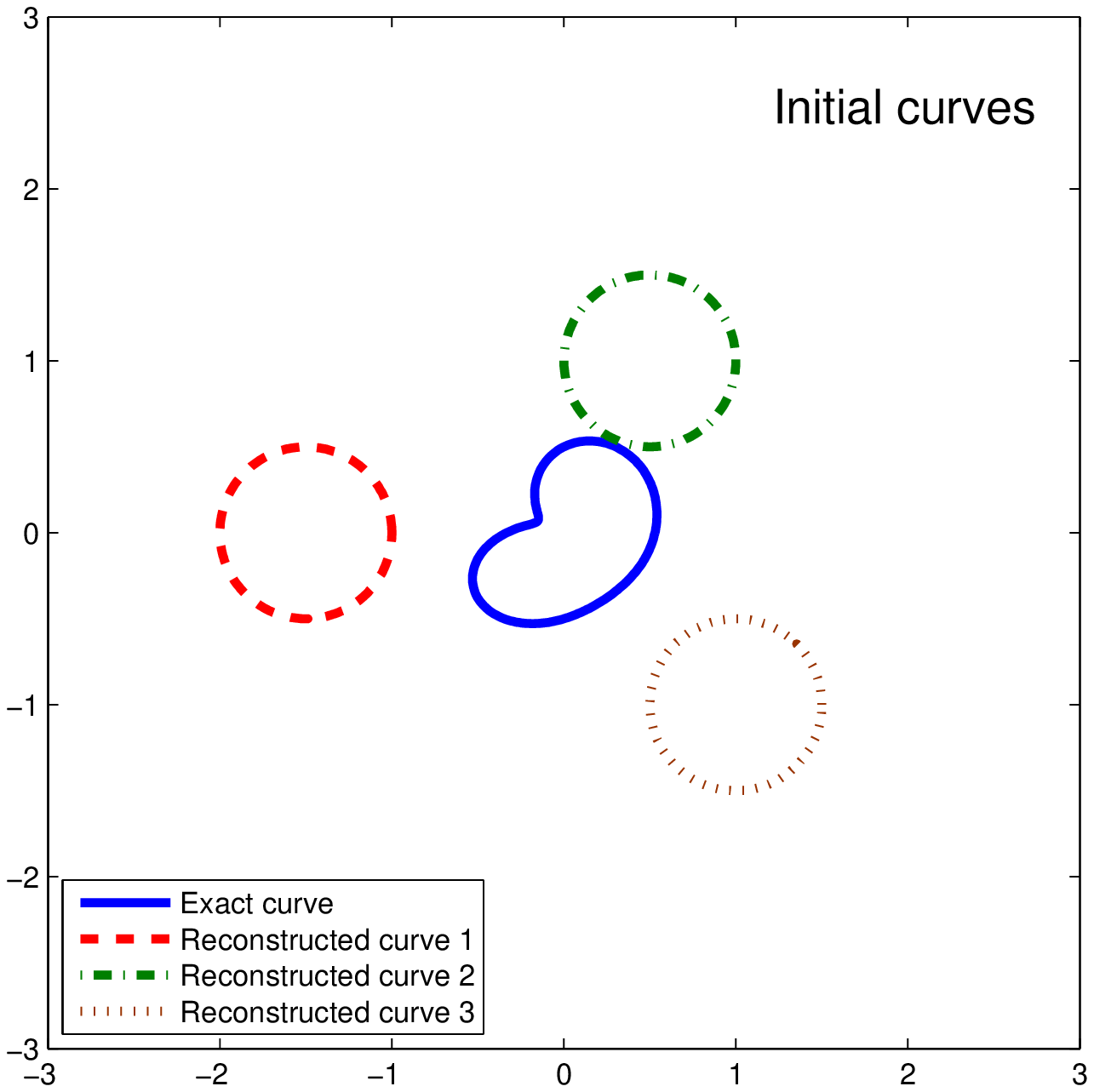}}
  \subfigure{\includegraphics[width=3in]{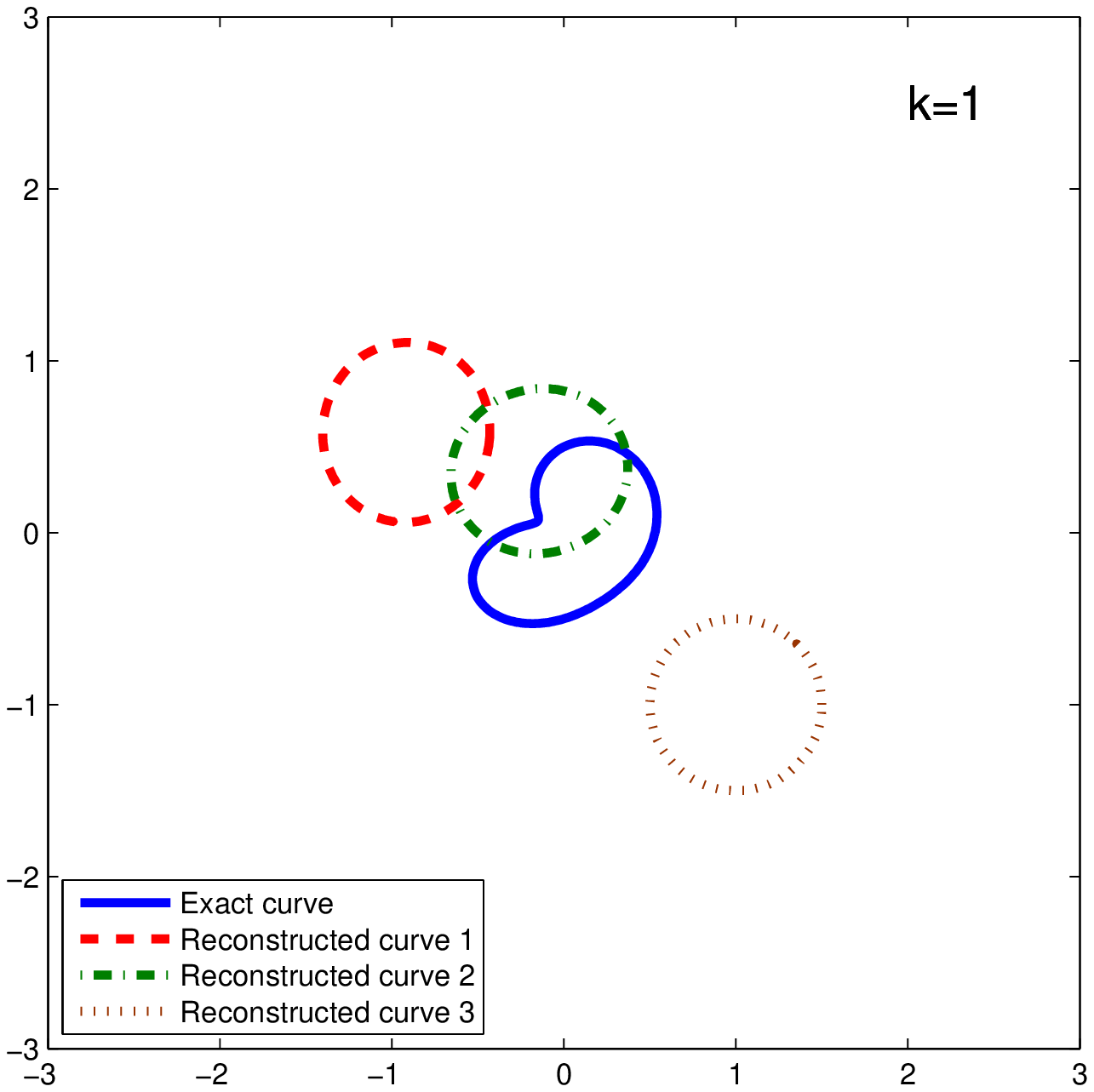}}
  \subfigure{\includegraphics[width=3in]{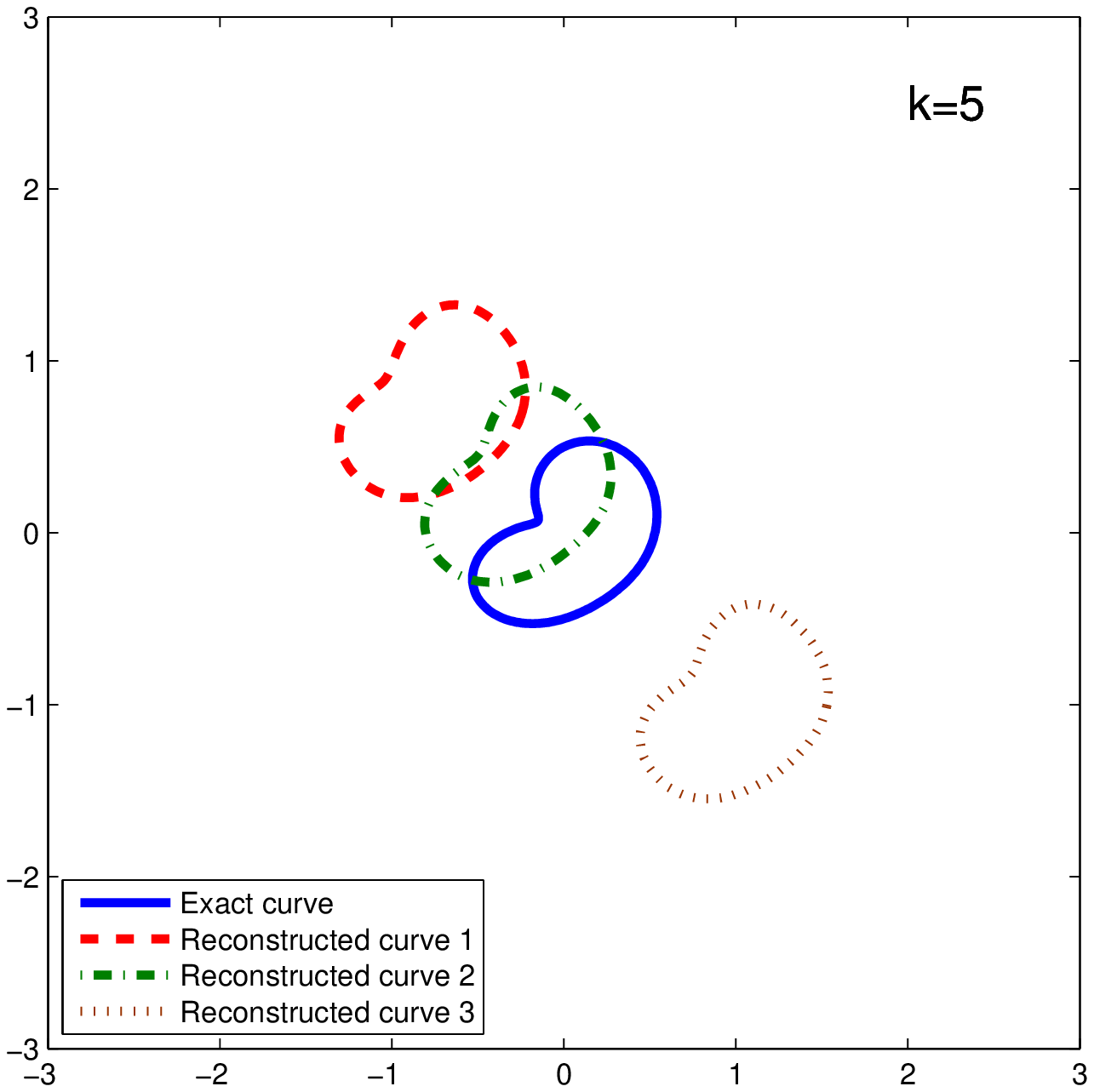}}
  \subfigure{\includegraphics[width=3in]{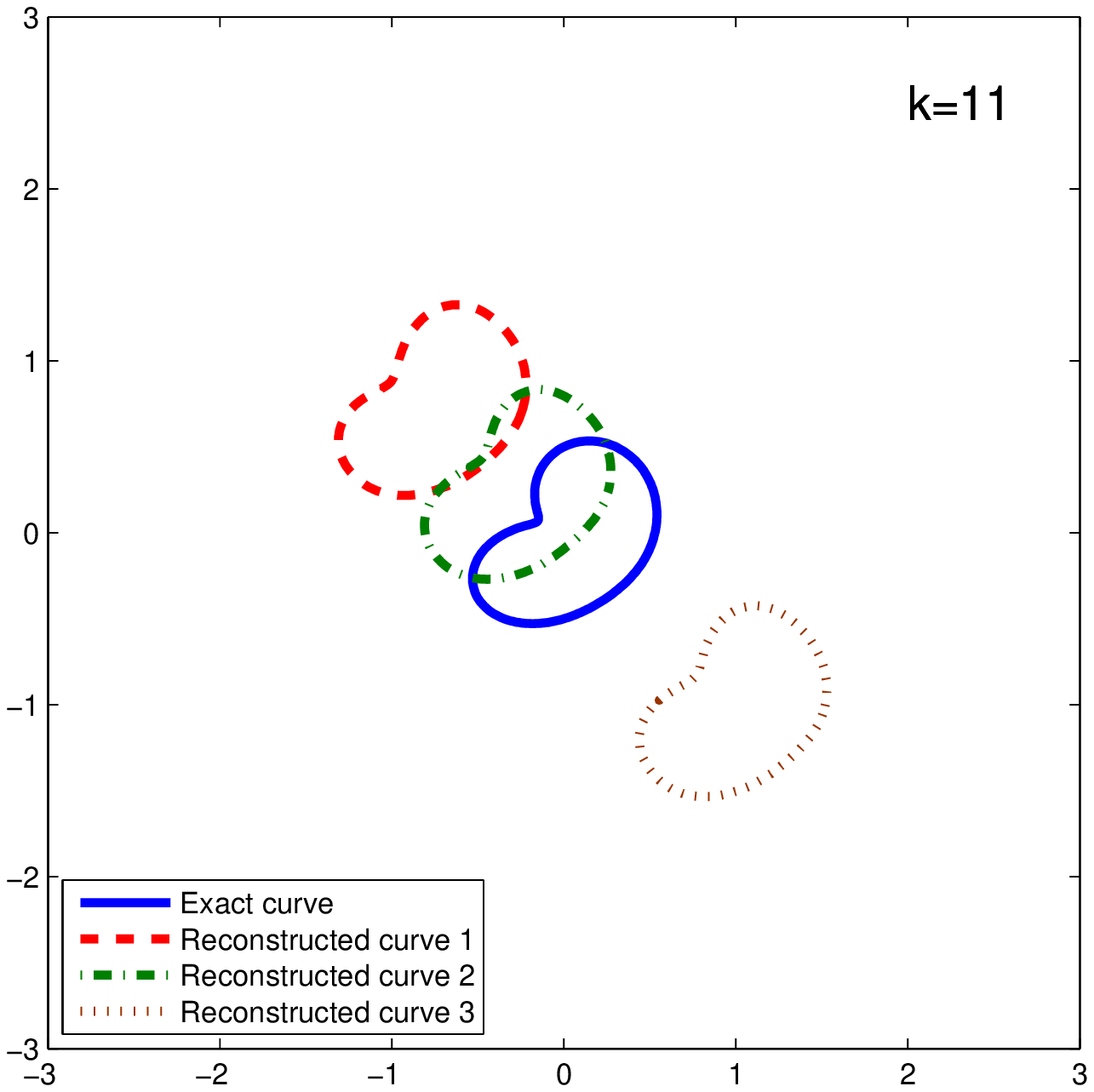}}
\caption{Shape reconstruction of a sound-soft, apple-shaped obstacle from the phaseless far-field data with $5\%$ noise,
corresponding to the superposition of two plane waves as the incident field $u^i=u^i(x;d_1,d_2,k)$
with multiple frequencies and two different directions $d_1=(1,0)$ and $d_2=(0,-1)$.
The initial guesses and the reconstructed obstacles at $k=1,5,11$ are presented.
}\label{fig3}
\end{figure}

\begin{figure}[htbp]
  \centering
  \subfigure{\includegraphics[width=3in]{pic/example1/initial_curves.eps}}
  \subfigure{\includegraphics[width=3in]{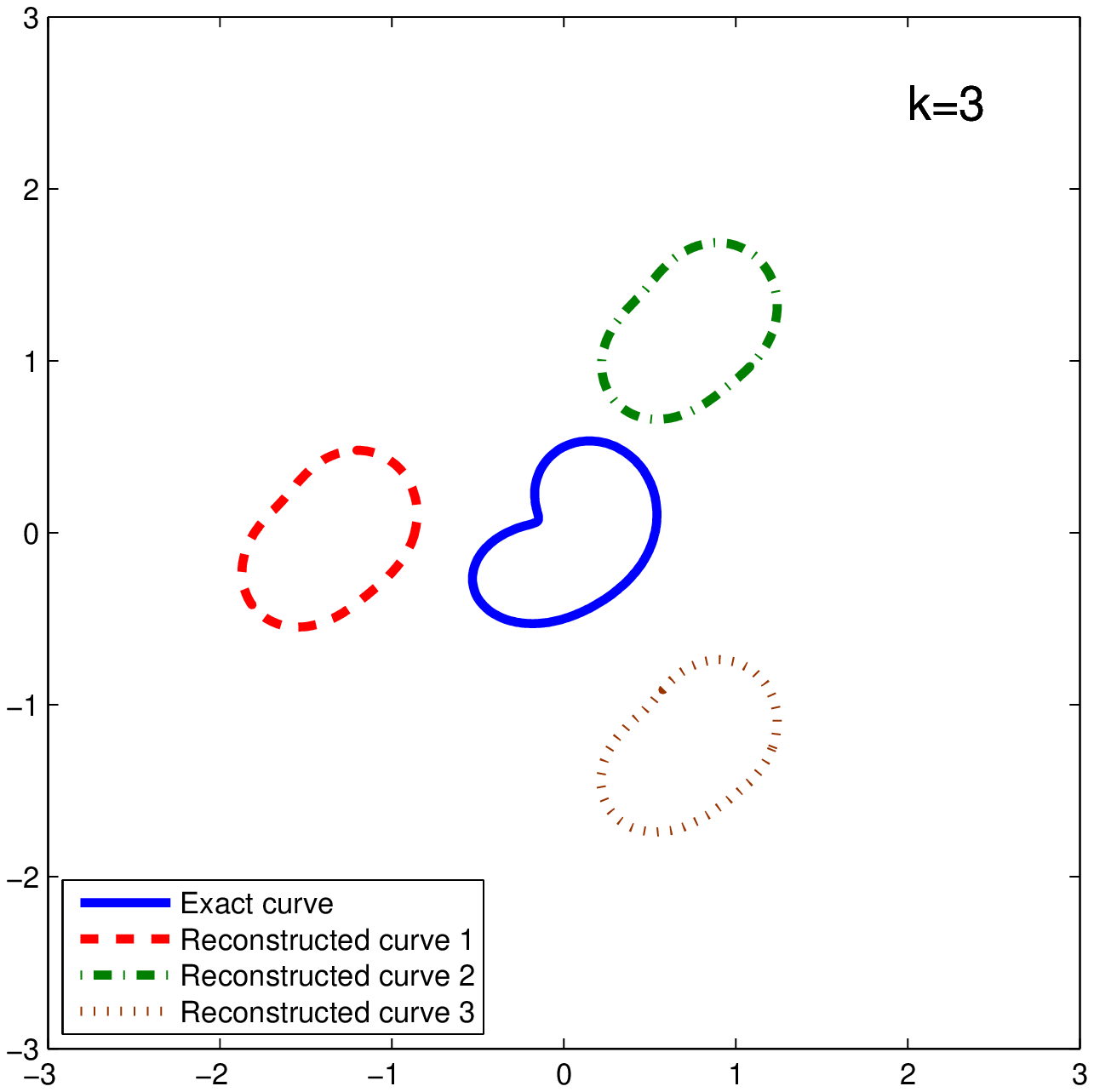}}
\caption{Shape reconstruction of a sound-soft, apple-shaped obstacle from the phaseless far-field data
with $5\%$ noise, corresponding to the incident fields $u^i=u^i(x;d_{1l},d_{2l},k)$, $l=1,2$,
with only one wave number $k=3$ and two different sets of incident directions $d_{11}=(1,0), d_{21}=(-1/2,\sqrt{3}/2)$
and $d_{12}=(1,0),d_{22}=(-1/2,-\sqrt{3}/2)$.
}\label{fig10}
\end{figure}

\textbf{Example 2: Location and shape reconstruction of a sound-soft obstacle.}

We consider the inverse problem (IP) for the same obstacle as in Example 1.
Again, we choose the same three initial guesses of the obstacle as in Example 1.
We use $5\%$ noisy phaseless far-field data generated by the incident fields
$u^i=u^i(x;d_{1l},d_{2l},k)$, $l=1,2,$ with two different sets of directions $d_{11}=(1,0), d_{21}=(-1/2,\sqrt{3}/2)$
and $d_{12}=(1,0),d_{22}=(-1/2,-\sqrt{3}/2)$.
Figure \ref{fig4} presents the initial and reconstructed curves at $k=1,5,11$.

From Figures \ref{fig10} and \ref{fig4} it is seen that
by using multi-frequency measured data both the location and the shape of the obstacle
are accurately reconstructed with all the three different initial guesses, as expected in Remark \ref{r2}.

\begin{figure}[htbp]
  \centering
  \subfigure{\includegraphics[width=3in]{pic/example1/initial_curves.eps}}
  \subfigure{\includegraphics[width=3in]{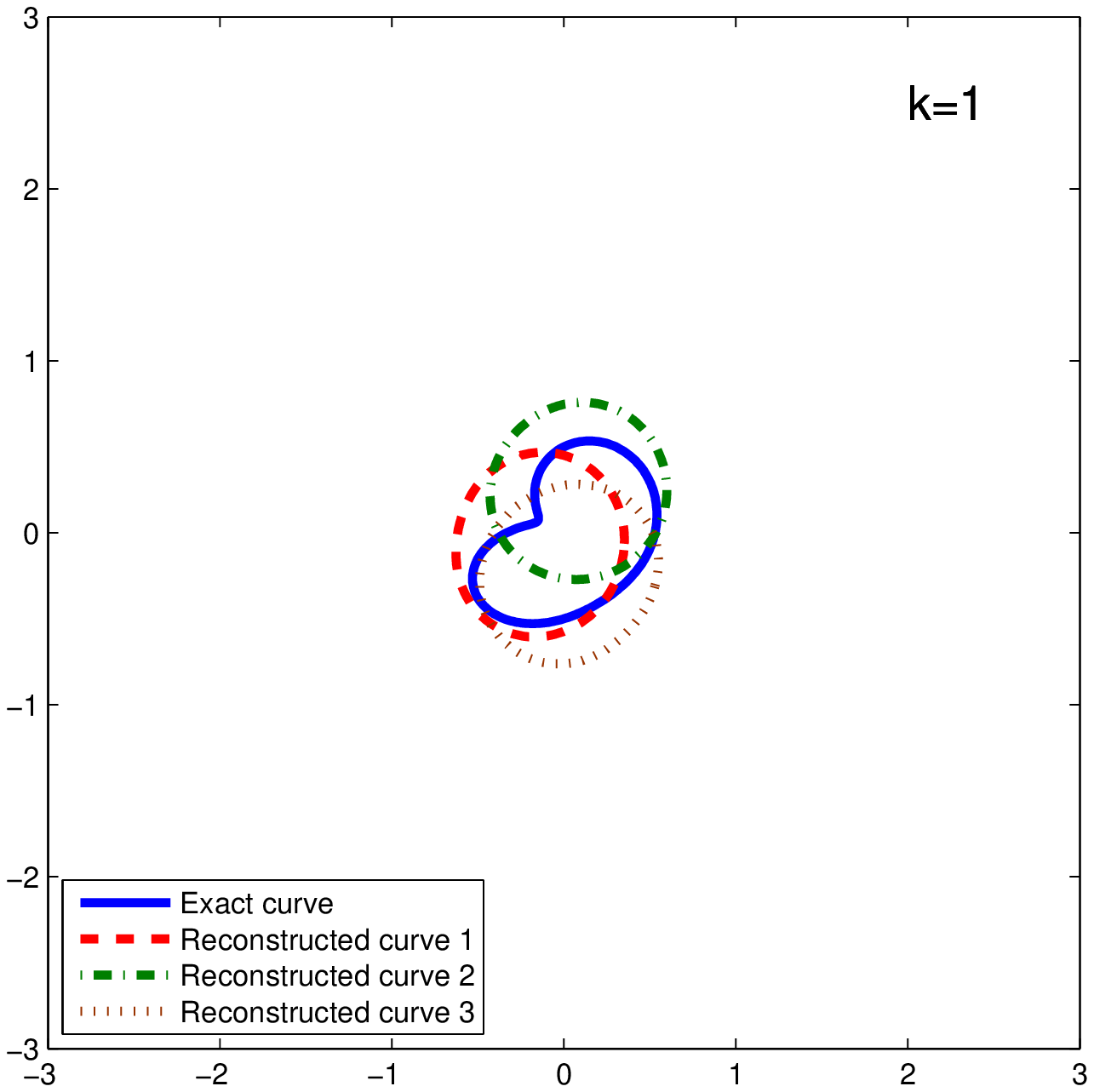}}
  \subfigure{\includegraphics[width=3in]{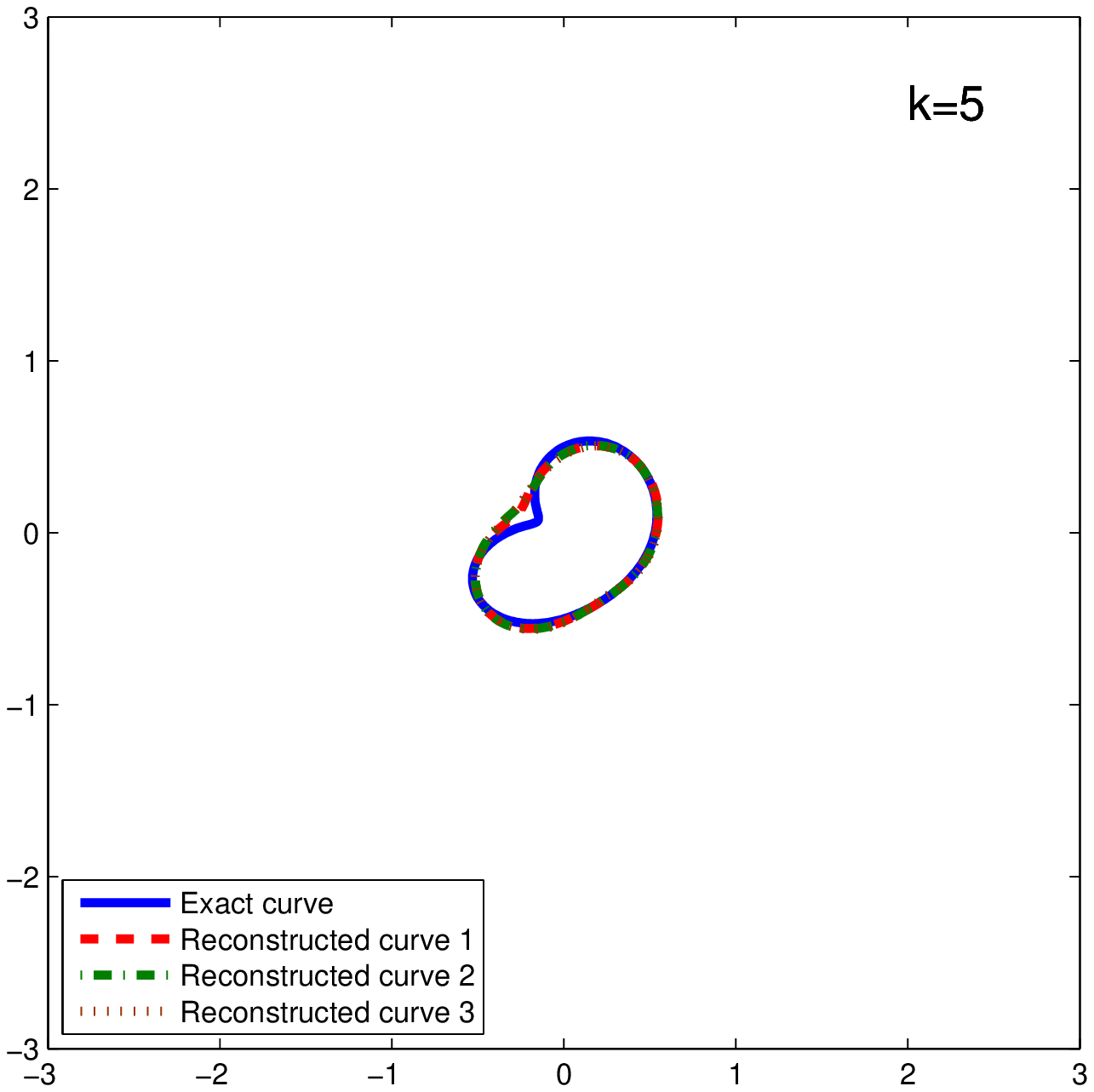}}
  \subfigure{\includegraphics[width=3in]{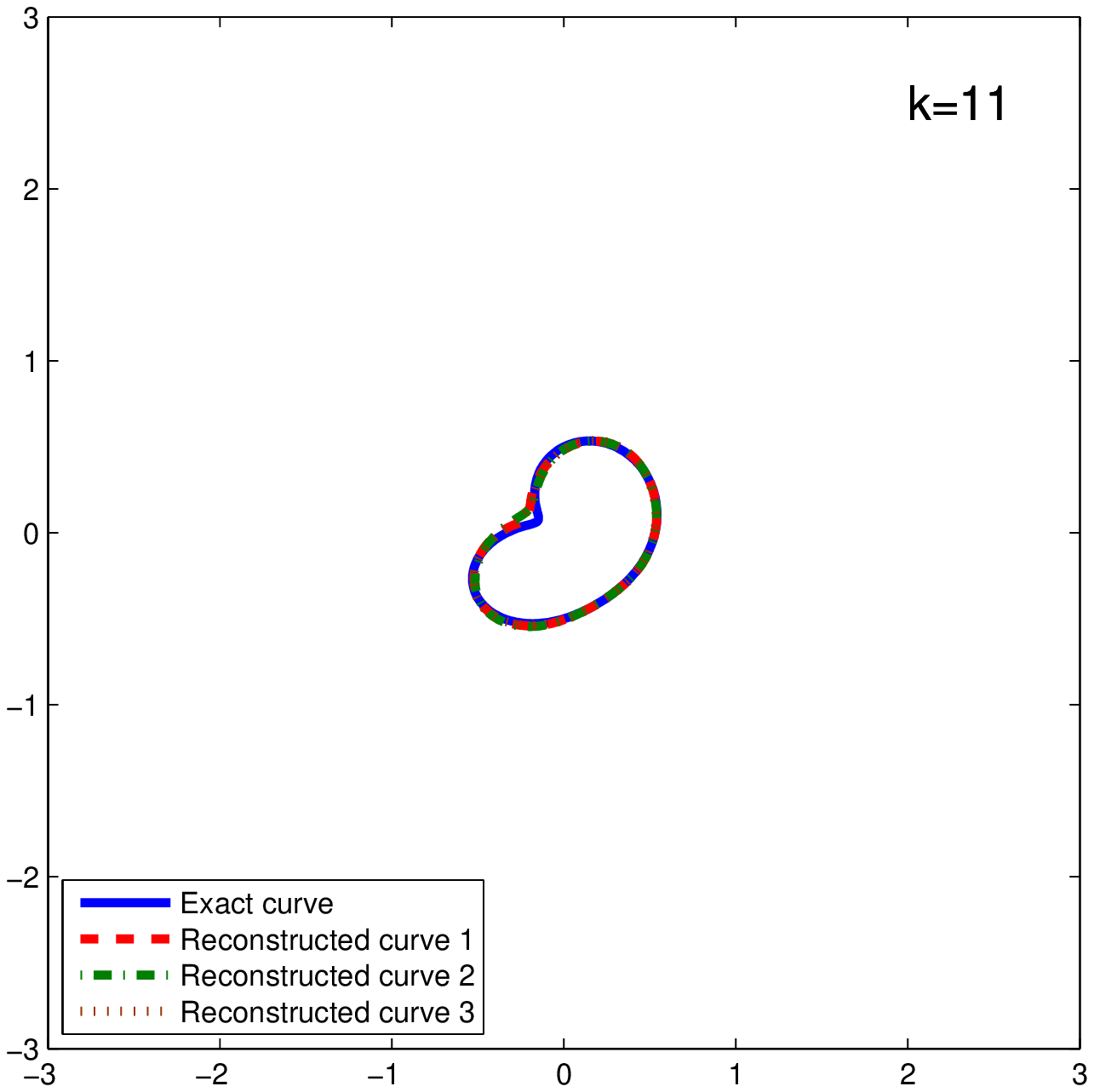}}
\caption{Location and shape reconstruction of a sound-soft, apple-shaped obstacle from the phaseless far-field data
with $5\%$ noise, corresponding to the incident fields $u^i=u^i(x;d_{1l},d_{2l},k)$, $l=1,2,$ with multiple frequencies
and two different sets of directions $d_{11}=(1,0), d_{21}=(-1/2,\sqrt{3}/2)$
and $d_{12}=(1,0),d_{22}=(-1/2,-\sqrt{3}/2)$.
The initial curves and the reconstructed obstacles at $k=1,5,11$ are presented.
}\label{fig4}
\end{figure}

\textbf{Example 3: Location and shape reconstruction of a sound-hard obstacle.}

We now consider the inverse problem (IP) with a sound-hard, kite-shaped obstacle.
The initial guess of the obstacle is taken to be a circle with radius $r_0=0.5$ and centered at $(1,1)$.
We use $5\%$ noisy phaseless far-field data generated by the incident fields
$u^i=u^i(x;d_{1l},d_{2l},k)$, $l=1,2,$ with multiple frequencies and two different sets of directions
$d_{11}=(1,0), d_{21}=(0,1)$ and $d_{12}=(0,1),d_{22}=(-1,0)$.
Figure \ref{fig6} presents the initial curve and the reconstructed curves at $k=1,5,11$.

From Figure \ref{fig6} it can be seen that both the location and the shape of the obstacle
are satisfactorily reconstructed.

\begin{figure}[htbp]
  \centering
  \subfigure{\includegraphics[width=3in]{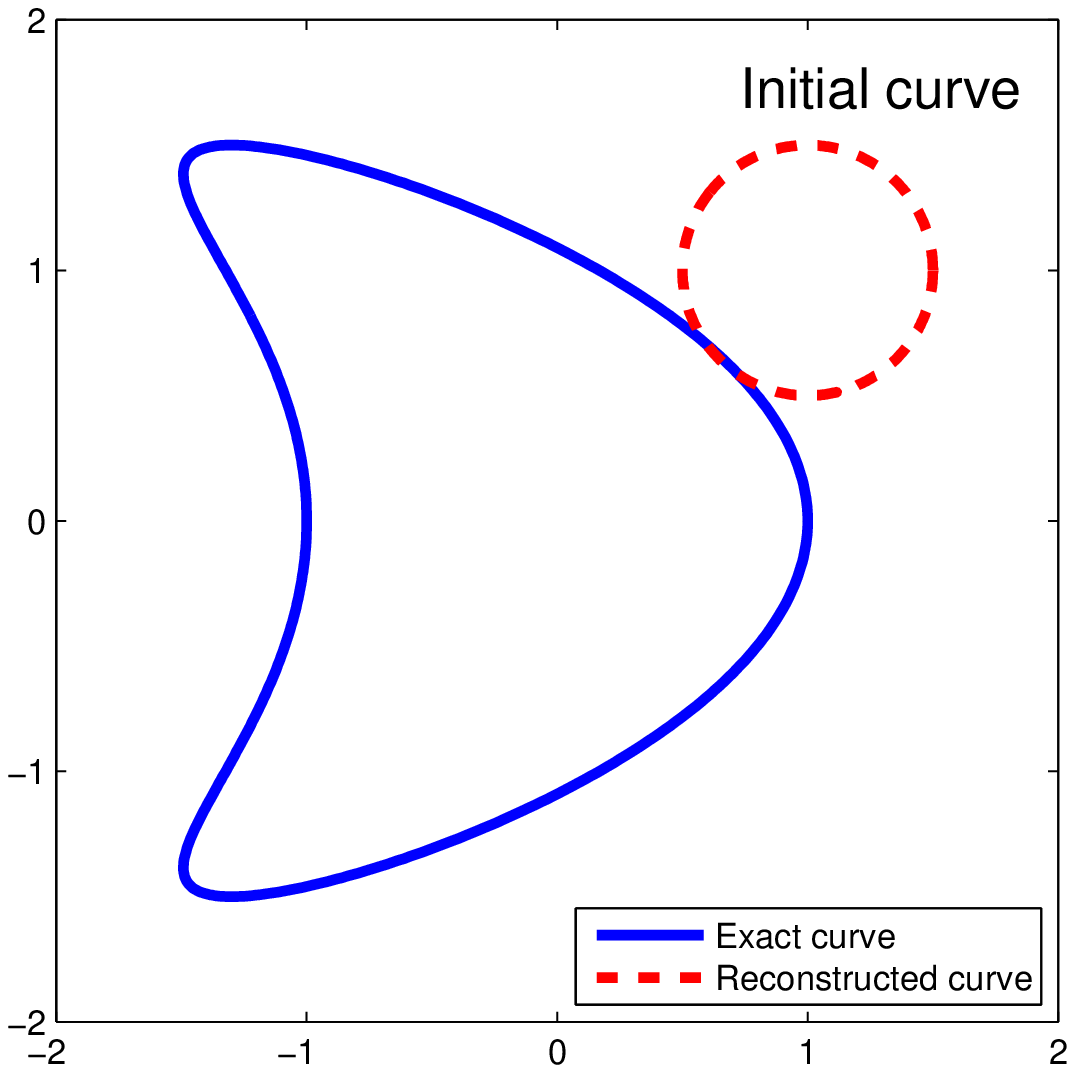}}
  \subfigure{\includegraphics[width=3in]{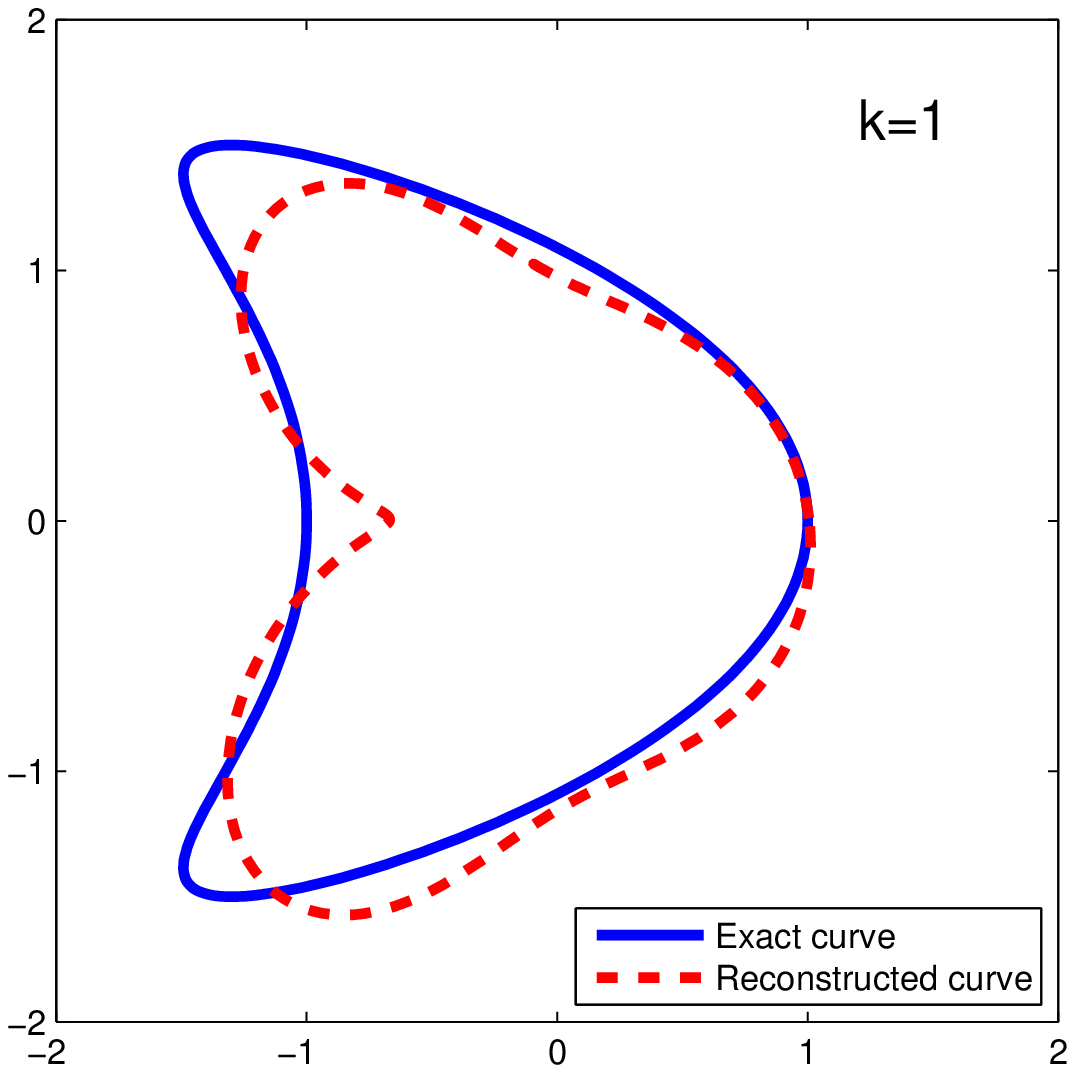}}
  \subfigure{\includegraphics[width=3in]{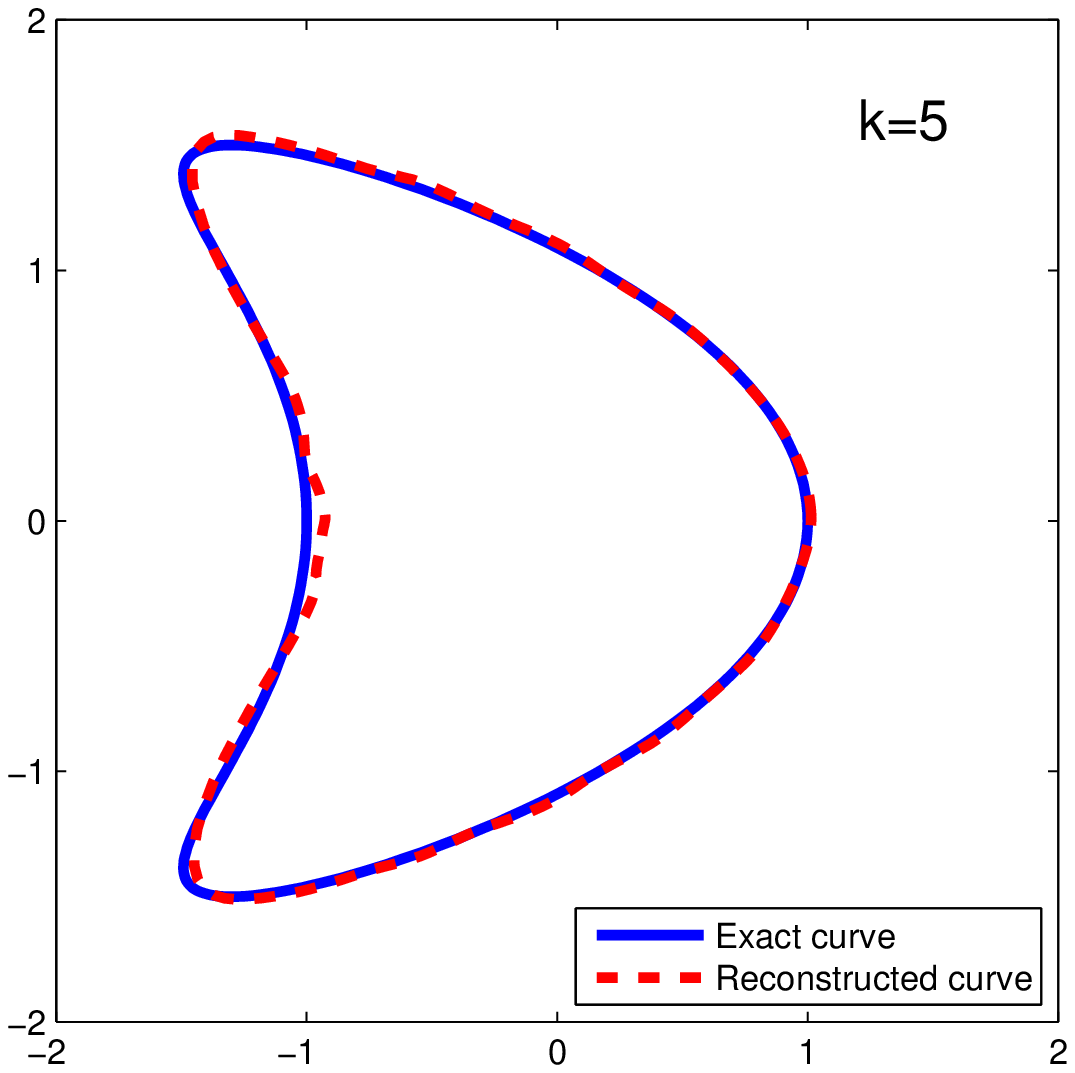}}
  \subfigure{\includegraphics[width=3in]{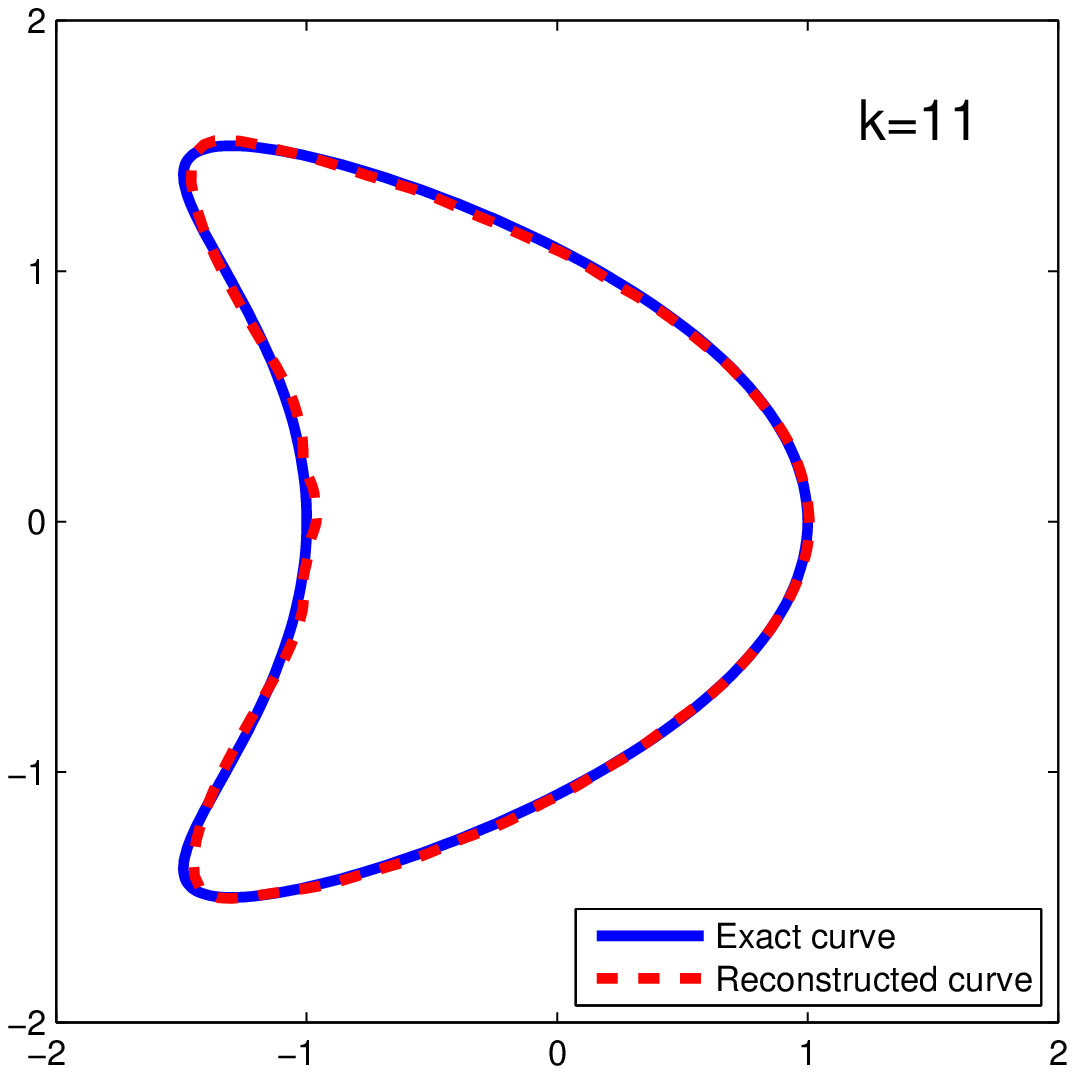}}
\caption{Location and shape reconstruction of a sound-hard, kite-shaped obstacle from the phaseless far-field data
with $5\%$ noise, corresponding to the incident fields $u^i=u^i(x;d_{1l},d_{2l},k)$, $l=1,2,$ with multiple frequencies
and two different sets of directions $d_{11}=(1,0), d_{21}=(0,1)$ and $d_{12}=(0,1),d_{22}=(-1,0)$.
The initial curve and the reconstructed obstacles at $k=1,5,11$ are presented.
}\label{fig6}
\end{figure}

\textbf{Example 4: Location and shape reconstruction of a penetrable obstacle.}

We consider the inverse problem (IP) with a penetrable, rounded triangle-shaped obstacle
with the refractive index $n=0.64$ and the transmission constant $\la=1.2$.
The initial guess of the obstacle is chosen to be a circle with radius $r_0=1.5$ and centered at $(-1,1)$.
The initial guess of the transmission constant $\la$ is taken to be $\la^{app}=1$.
We use $5\%$ noisy phaseless far-field data generated by
the incident fields $u^i=u^i(x;d_{1l},d_{2l},k)$, $l=1,2,3,4,$ with multiple frequencies and four different
sets of directions $d_{11}=(1,0), d_{21}=(0,1)$, $d_{12}=(0,1), d_{22}=(-1,0)$, $d_{13}=(-1,0), d_{23}=(0,-1)$,
$d_{14}=(0,-1),d_{24}=(1,0)$.
Figure \ref{fig8} presents the initial curve and the reconstructed curves at $k=1,5,11,$ respectively.
The reconstructed value of $\la$ is $1.263$.

From Figure \ref{fig8} it is found that both the location and the shape of the obstacle are
reconstructed very well. Further, the reconstructed values of $\la$ are also satisfactory.

\begin{figure}[htbp]
  \centering
  \subfigure{\includegraphics[width=3in]{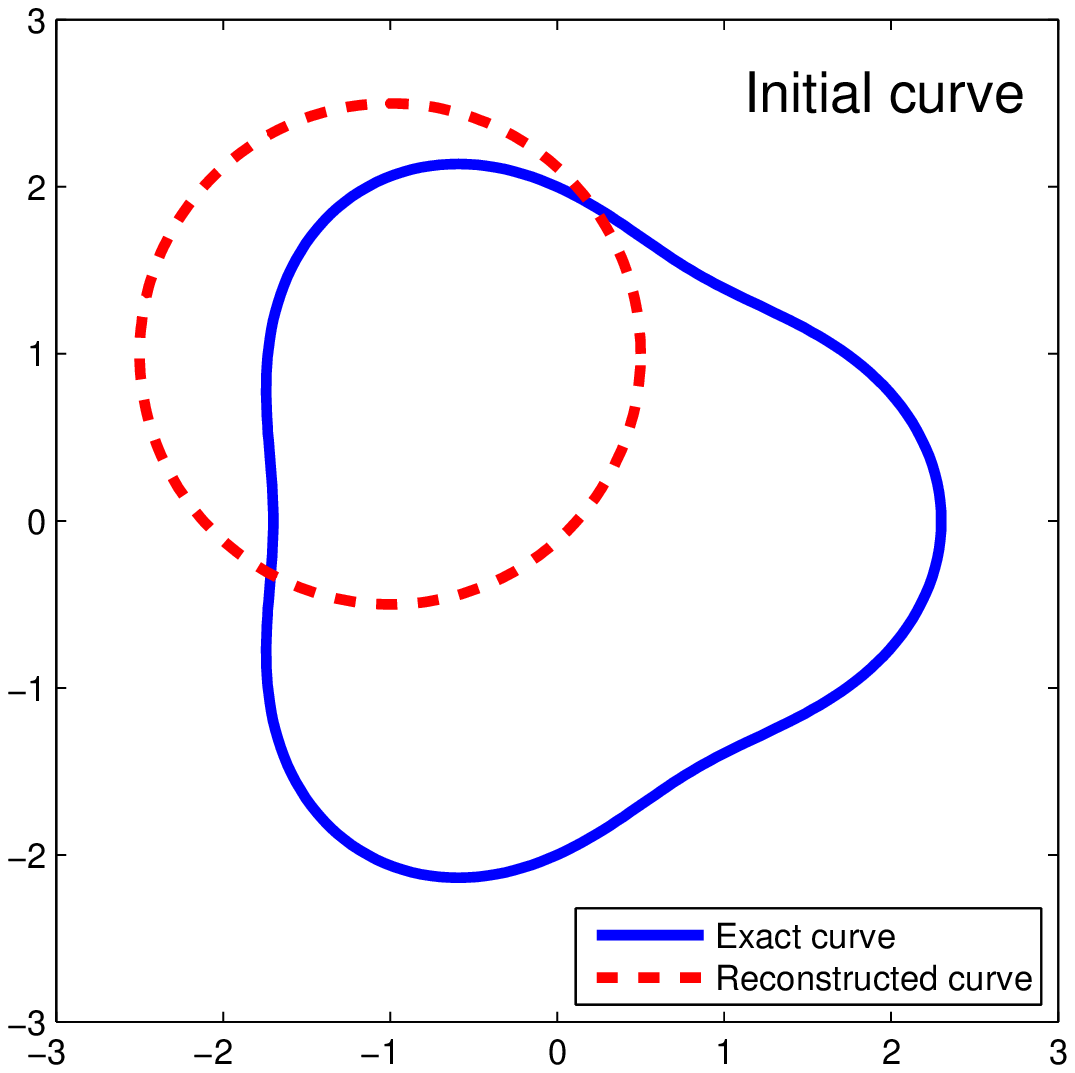}}
  \subfigure{\includegraphics[width=3in]{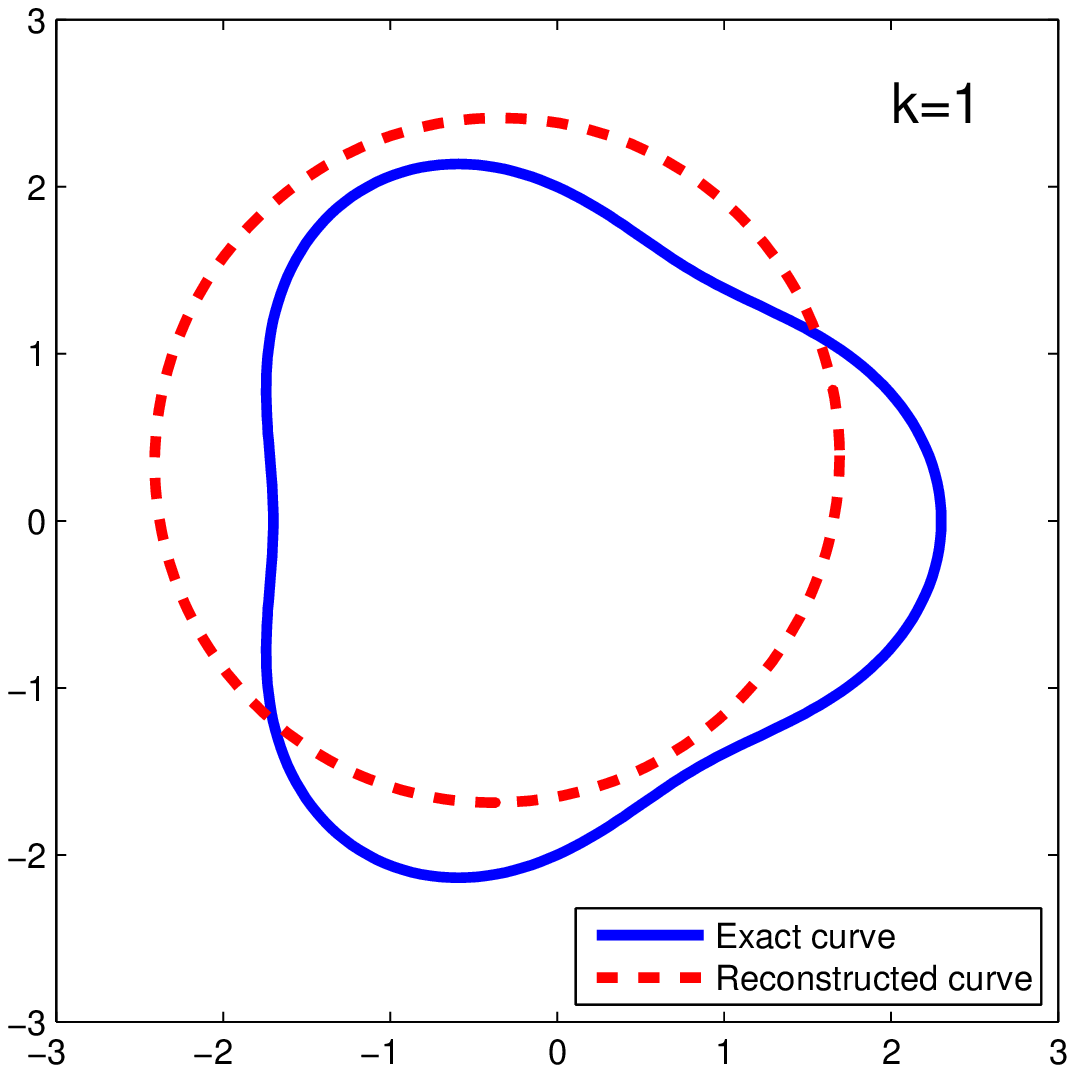}}
  \subfigure{\includegraphics[width=3in]{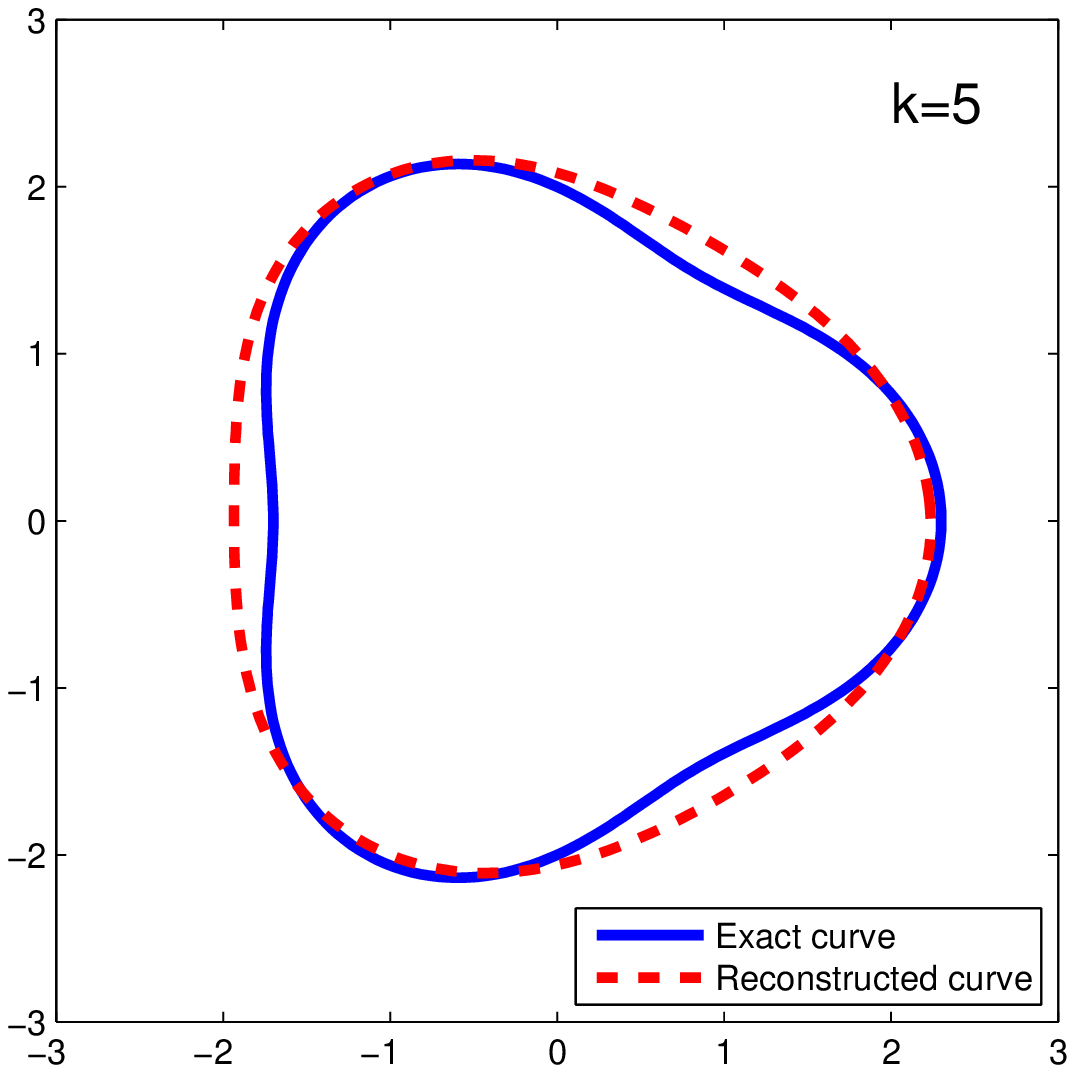}}
  \subfigure{\includegraphics[width=3in]{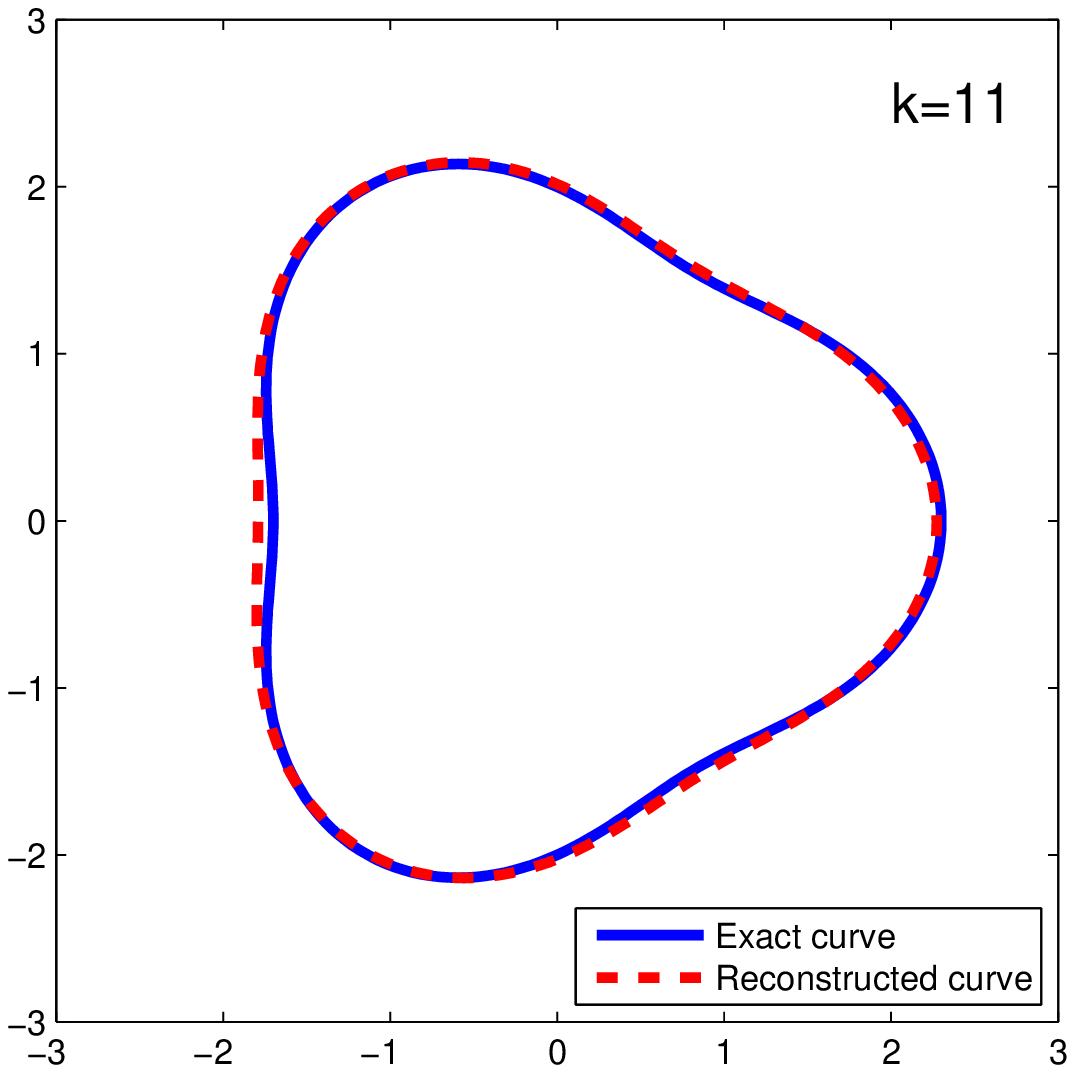}}
\caption{Location and shape reconstruction of a penetrable, rounded triangle-shaped obstacle from the phaseless
far-field data with $5\%$ noise, corresponding to the incident fields $u^i=u^i(x;d_{1l},d_{2l},k)$, $l=1,2,3,4$
with multiple frequencies and four different sets of directions $d_{11}=(1,0), d_{21}=(0,1)$,
$d_{12}=(0,1), d_{22}=(-1,0)$, $d_{13}=(-1,0), d_{23}=(0,-1)$ and $d_{14}=(0,-1),d_{24}=(1,0)$.
The initial curve and the reconstructed obstacles at $k=1,5,11$ are presented.
The transmission constant $\la$ and its reconstructed value are $1.2$ and $1.263$, respectively.
}\label{fig8}
\end{figure}

From the above examples we found that the translation invariance property of the phaseless far-field pattern
can be broken down by using a superposition of two plane waves as the incident field, in conjunction with
multiple frequencies. It is further found that the inverse problem (IP), based on this approach,
can be solved by the proposed recursive Newton-type iteration algorithm in frequencies for reconstructing
both the location and the shape of the obstacles from the multi-frequency phaseless far-field data.

\section{Conclusions and future works}

In this paper we devised a new approach to break down the translation invariance property of the modulus of
the far-field pattern (or phaseless far-field pattern) in inverse obstacle scattering.
Precisely, we proposed to use superpositions of two plane waves rather than one plane wave as the incident fields
and proved that the translation invariance property of the phaseless far-field pattern is broken by using
such incident fields with all wave numbers in a finite interval.
We developed a recursive Newton-type iteration algorithm in frequencies to recover both the location and the shape
of the obstacle simultaneously from multi-frequency phaseless far-field data.
Numerical examples have also been conducted to demonstrate that our approach is valid and the inverse algorithm
is effective. The approach has been extended to inverse scattering by locally rough surfaces in \cite{ZZ17}.
It is expected that the approach can be extended to the three-dimensional case and to other inverse
scattering problems such as inverse electromagnetic scattering.
{
It should be also mentioned that in practice the use of phaseless measurements is to stabilize the imaging with
respect to medium noise. Medium noise corrupts much more the phase of the scattered wave than its amplitude.
Our proposed method cross-correlates two signals obtained from incident different directions. This should be quite
stable with respect to medium noise. A similar idea has be implemented in a completely different context in \cite{AGJN13}.
}

On the other hand, radar cross section (RCS) is an important measure for radar systems to detect a target.
Intuitively, RCS measures how much energy reflected back from the infinity compared to the incident wave, and
mathematically, it measures the intensity of the far-field pattern of the scattered field (the square of the
modulus of the far-field pattern). Thus, in this paper, we actually developed an iterative algorithm to recover
the scattering obstacle by multi-frequency RCS.
In view of the results presented in this paper, we have also developed a fast imaging algorithm to recover
scattering obstacles by intensity-only far-field data (or RCS) at a fixed frequency \cite{ZZ16}.

\section*{Acknowledgements}

This work is partly supported by the NNSF of China grants 91430102, 91630309, 61379093 and 11501558
and the National Center for Mathematics and Interdisciplinary Sciences, CAS.
We thank Professor Rainer Kress at University of G\"ottingen, Germany for pointing out a
counterexample to the previous version of Theorem \ref{thm2}, leading to the present improved
version of Theorem \ref{thm2}. We also thank Dr. Xiaodong Liu and Dr. Jiaqing Yang for helpful discussions.

\end{document}